\documentclass[smallcondensed]{svjour3}
\usepackage[width=16cm]{geometry}
\usepackage{graphicx}
\usepackage{amssymb}
\usepackage{amsmath}
\usepackage{units}
\usepackage{stmaryrd}
\usepackage{subfigure}
\usepackage{multirow}
\usepackage{mathtools}
\usepackage[usenames,dvipsnames]{xcolor}

\newcommand{\thr}{\theta^+}
\newcommand{\thl}{\theta^-}

\newcommand*{\boxedcolor}{red}
\makeatletter
\renewcommand{\boxed}[1]{\textcolor{\boxedcolor}{%
  \fbox{\normalcolor\m@th$\displaystyle#1$}}}
\makeatother

\renewcommand{\leq}{\leqslant}

\renewcommand{\geq}{\geqslant}

\newcommand{\Ro}{\mathbb{R}}

\newcommand{\ud}{\text{d}}

\newcommand{\dd}[2]{\frac{\ud #1}{\ud #2}}

\newcommand{\avg}[1]{\overline{#1}}

\newcommand{\hf}{{\unitfrac{1}{2}}}
\newcommand{\imo}{{i-1}}
\newcommand{\ipo}{{i+1}}
\newcommand{\jph}{{j+\hf}}

\newcommand{\iph}{{i+\hf}}
\newcommand{\imh}{{i-\hf}}

\newcommand{\wt}{\widetilde{w}}

\begin{document}
\title{A sign preserving WENO reconstruction method\thanks{DR was supported by AIRBUS Group Corporate Foundation Chair in Mathematics of Complex Systems, established in TIFR/ICTS, Bangalore. }}
\author{Ulrik S. Fjordholm \and Deep Ray}
\institute{U. S. Fjordholm \at
              Department of Mathematical Sciences, NTNU, Trondheim 7491, Norway \\
              \email{ulrik.fjordholm@math.ntnu.no}
           \and
           D. Ray \at
              TIFR - Centre for Applicable Mathematics, Bangalore
              \email{deep.day@gmail.com}
}

\maketitle

\begin{abstract}
We propose a third-order WENO reconstruction which satisfies the \emph{sign property}, required for constructing high resolution entropy stable finite difference scheme for conservation laws. The reconstruction technique, which is termed as SP-WENO, is endowed with additional properties making it a more robust option compared to ENO schemes of the same order. The performance of the proposed reconstruction is demonstrated via a series of numerical experiments for linear and nonlinear scalar conservation laws. The scheme is easily extended to multi-dimensional conservation laws.
\end{abstract}


\maketitle

\section{Introduction}\label{sec:intro}
High-resolution numerical methods have been of key importance in the past few decades. Hyperbolic system of conservation laws are among the class of problems that require such efficient and accurate schemes methods. Conservative finite difference (finite volume) methods, in which the computational domain is divided into control volumes and a discrete version of the conservation law imposed on each control volume, are very popular. In these methods, point values (cell-averages) in each control volume are evolved in time using suitable time integration techniques, such as strong stability preserving Runge-Kutta methods \cite{GOT01}. Higher-order methods are obtained by suitably reconstructing the solution in each control volume. Reconstruction using total variation diminishing (TVD) limiters have good non-linear stability properties, but can lead to clipping of smooth extrema \cite{OSHER84b,OSHER88}. Essentially non-oscillatory (ENO) reconstruction methods were first introduced by Harten et al.\ \cite{HARTEN87}. The main idea is to choose the {\em smoothest} stencil among a number of candidate stencils, to reconstruct the solution at the cell-interfaces to be used in the numerical flux. ENO schemes have been quite successful in practice, but can show deterioration in accuracy due to selection of unstable stencils \cite{ROGER90}.

Weighted ENO (WENO) schemes \cite{LIU94,JS96} were proposed as an improvement over ENO schemes. The basic idea of WENO is to take a convex combination of all $k$ polynomials involved in the $k$-th order ENO (ENO-k) approximation at an interface, and obtain a $(2k-1)$-th order approximation of the interface value. The weights are chosen so as to give the least weight to stencils containing discontinuities. It has been shown in \cite{SHU98} that WENO schemes do not suffer from accuracy deterioration faced by ENO schemes. 

Scalar conservation laws have been studied and analysed in great detail. It is known that the solutions can develop discontinuities, despite having smooth initial data \cite{DAFER10}. Thus, the solutions must be interpreted in a weak (distributional) sense. Additional conditions known as {\em entropy conditions} need to be imposed to single out a physically relevant solutions. Scalar conservation laws have been shown to have unique entropy weak solutions \cite{KRUZ70}. Hence, it becomes important to construct numerical methods whose solutions converge to the entropy solution.

For scalar conservation laws in one dimension, monotone schemes have been shown to be total variation diminishing (TVD) and satisfy the entropy condition \cite{CRANDALL80}. {\em E-schemes} have been designed in \cite{OSHER84a} to preserve a discrete version of the entropy condition. However, E-schemes --- and in particular monotone schemes --- are at most first-order accurate. Second-order schemes using flux-limiters were introduced in \cite{SWEBY84}. Tadmor \cite{TADMOR03} proposed a new method of constructing entropy stable schemes for conservation laws, which consists of two components: i) constructing a second-order {\em entropy conservative} scheme that preserves entropy locally, ii) adding artificial dissipation to get entropy stability. Higher-order entropy conservative finite difference schemes have been constructed in \cite{LMR02}. The design of arbitrary-high order entropy stable schemes was proposed recently by Fjordholm et al.\ \cite{FMT12}. These so-called TeCNO schemes combine high-order entropy conservative fluxes with high-order numerical diffusion operators, based on piecewise polynomial reconstruction. The reconstructions have to satisfy a {\em sign property} at each cell interface to ensure entropy stability. This means that the jump in the reconstructed values at every cell interface must have the same sign as the jump in the corresponding cell values. It was shown in \cite{FMT13} that the standard ENO reconstruction procedure satisfies the sign property. 

To the best of our knowledge, the second order limited TVD reconstruction using the minmod limiter and the ENO method are the only reconstruction methods that satisfy the crucial sign property \cite{FMT12}. Existing WENO schemes do not satisfy the sign-property. The aim of the current paper is to construct a third-order WENO reconstruction method which satisfies the sign-property. This leads to a third-order accurate entropy stable scheme when used in conjunction with the TeCNO schemes. 

The rest of the paper is organized as follows. In Section \ref{sec:ent_frame} we briefly introduce the entropy framework for scalar conservation laws. Section \ref{sec:mesh_fdm} highlights the existing work on high-order entropy conservative and entropy stable schemes. The properties of a third-order sign preserving WENO scheme are discussed in Section \ref{sec:weno}, with explicit WENO weights constructed in Section \ref{sec:sp_weno}. We show that the proposed reconstruction method is stable in the sense that the reconstructed jumps can be at most twice as large as the original values. Several numerical results are presented for one-dimensional linear advection and the inviscid Burgers equation in Section \ref{sec:numerical_results}. We end with concluding remarks in Section \ref{sec:conclusion}.

\section{Entropy framework}\label{sec:ent_frame}
Consider the following Cauchy problem for a scalar conservation law
\begin{equation}\label{eqn:conlaw}
\begin{aligned}
 \partial_t u + \partial_x f(u) = 0& \qquad \forall \ (x,t)\in \Ro \times \Ro^+ \\
 u(x,0) = u_0(x) & \qquad \forall\ x \in \Ro.
 \end{aligned}
\end{equation}
In the above problem, $u$ is the conserved variable with a smooth flux $f(u)$. Assume that \eqref{eqn:conlaw} is equipped with a convex {\em entropy function} $\eta(u)$ and an {\em entropy flux} $q(u)$ such that $
 q^\prime(u) = \eta^\prime (u) f^\prime(u)$ holds.
Multiplying \eqref{eqn:conlaw} with the {\em entropy variable} $v(u) = \eta^\prime(u)$, results in an additional conservation law for smooth solutions:
\begin{equation}\label{eqn:entropyeq}
 \partial_t \eta (u) + \partial_x q(u) = 0 .
\end{equation}
However, for discontinuous solutions, entropy should be dissipated at shocks, and hence one imposes the entropy condition
\begin{equation}\label{eqn:entropyineq}
 \partial_t \eta (u) + \partial_x q(u) \leq 0
\end{equation}
which is understood in the sense of distributions. A weak solution of \eqref{eqn:conlaw} is called an {\em entropy solution} if \eqref{eqn:entropyineq} holds. Formally integrating \eqref{eqn:entropyineq} in space and ignoring the boundary terms by assuming periodic or no-inflow boundary conditions, we get
\begin{equation}\label{eqn:globalentropyineq}
\dd{}{t} \int \limits_{\Ro} \eta (u) \ud x  \leq 0 \quad \implies \quad  \int \limits_{\Ro} \eta (u(x,t)) \ud x  \leq  \int \limits_{\Ro} \eta (u_0(x)) \ud x \qquad \forall \ t > 0 .
\end{equation}
As $\eta$ is convex, the above entropy bound gives rise to an a priori estimate on the solution of \eqref{eqn:conlaw} in suitable $L^p$ spaces \cite{DAFER10}. 

For scalar conservation laws, any convex function $\eta(u)$ can serve as an entropy function, with the corresponding entropy flux chosen as 
\[
q(u) = \int^u \eta'(z) f'(z) \ud z.
\]
This idea was been exploited to prove stability and uniqueness of entropy solutions of scalar conservation laws \cite{KRUZ70}.

\section{Mesh and finite difference scheme}\label{sec:mesh_fdm}
We discretise the domain using disjoint intervals $I_i = [x_\imh,x_\iph)$ of uniform length $x_\iph - x_\imh \equiv h$. We use the notation $x_i$ to denote the center of the interval $I_i$. A generic semi-discrete finite difference scheme for \eqref{eqn:conlaw} is given by
\begin{equation}\label{eqn:fdm}
\dd{ u_i}{t} + \frac{1}{h} \left( F_\iph - F_\imh \right) = 0.
\end{equation}
Here, $u_i(t) = u(x_i,t)$ is the point values of the solution at the cell centre $x_i$, while $F_\iph(t) = F(u_i(t), u_\ipo(t))$ is a consistent and conservative approximation of the flux at the cell interface $x_\iph$.

\subsection{Entropy conservative schemes}
As mentioned in the introduction, we are interested in looking at entropy stable schemes to approximate \eqref{eqn:conlaw}.  Following the approach described by Tadmor \cite{TADMOR03}, we first consider an entropy conservative scheme.

\begin{definition}
The numerical scheme \eqref{eqn:fdm} is said to be {\em entropy conservative} if it satisfies the discrete entropy relation
\begin{equation}\label{eqn:ECrel}
 \normalfont
 \dd{\eta(u_i)}{t} + \frac{1}{h} \bigl(q^*_\iph - q^*_\imh\bigr) = 0
\end{equation}
where $q^*_\iph$ is a consistent numerical entropy flux.
\end{definition}
We introduce the following notation:
\[
 \Delta (\cdotp)_{\iph} = (\cdotp)_\ipo - (\cdotp)_i, \qquad \avg{(\cdotp)}_{\iph} = \frac{(\cdotp)_\ipo +(\cdotp)_i}{2}.
\]
Moreover, we introduce the \emph{entropy potential}
\[
\Psi(u) := v(u) f(u) - q(u).
\]
The following theorem gives a sufficient condition for constructing entropy conservative fluxes.
\begin{theorem}[Tadmor \cite{TADMOR84}]\label{thm:ECflux}
 The numerical scheme \eqref{eqn:fdm} with the flux $F^{*}_\iph= F^*(u_i,u_\ipo)$ is entropy conservative if
 \begin{equation}\label{eqn:ECfluxcond}
  \Delta v_\iph F^*_\iph= \Delta \Psi_\iph.
 \end{equation}
Specifically, it satisfies \eqref{eqn:ECrel} with the consistent numerical entropy flux given by
 \begin{equation*}
  q^*_\iph = \overline{v}_\iph F^*_\iph - \avg{\Psi}_\iph.
 \end{equation*}
 Furthermore, the scheme is second order accurate in space.
\end{theorem}
For scalar conservation laws and a given entropy function $\eta(u)$, two-point entropy conservative fluxes are uniquely determined by
\[
F^*_\iph = \frac{\Delta \Psi_\iph}{ \Delta v_\iph}.
\]
\begin{example}
Consider the linear advection equation with flux $f(u) = cu$, where c is a constant. Choosing the square entropy $\eta(u) = u^2/2$ and the corresponding entropy flux as $q(u) = c u^2/2$, we get the entropy conservative flux
\begin{equation}\label{eqn:ecflux_linadv}
F^*_\iph = c\frac{\left(u_i + u_\ipo\right)}{2}.
\end{equation}
\end{example}

\begin{example}
Consider the Burgers equation with flux $f(u) = u^2/2$. The square entropy $\eta(u) = u^2/2$ and the corresponding entropy flux as $q(u) =  u^3/3$ results in  the entropy conservative flux
\begin{equation}\label{eqn:ecflux_burgers}
F^*_\iph =  \frac{\left(u_i^2 + u_\ipo^2 + u_i u_\ipo\right)}{6}.
\end{equation}
\end{example}

The entropy conservative schemes described above are only second order accurate \cite{TADMOR03}. However, the approach of LeFloch, Mercier and Rhode \cite{LMR02} can be used to used to construct higher order entropy conservative fluxes. The basic idea of their approach is to take appropriate linear combinations of existing second-order accurate entropy conservative fluxes 
\begin{equation}\label{eqn:EC_high}
F^{*,2p}_\iph = \sum \limits_{r=1}^{p} \alpha_r^p \sum \limits_{s=0}^{r-1} F^*(u_{i-s},u_{i-s+r})
\end{equation}
to obtain $2p$th-order accurate entropy conservative fluxes.
For instance, the fourth-order $(p=2)$ version of the entropy conservative flux has the expression
\begin{equation}\label{eqn:EC4}
F^{*,4} = \frac{4}{3} F^*(u_i, u_\ipo) - \frac{1}{6} \bigl(F^*(u_{\imo}, u_{\ipo}) + F^*(u_i, u_{i+2})\bigr).
\end{equation}
\begin{remark}
Since any convex choice for the entropy function $\eta(u)$ works for scalar conservation laws, we shall adhere to the choice $\eta(u) = u^2/2$ for the remainder of this paper. The case of an arbitrary convex entropy is in no way more difficult, so this choice is merely for the sake of simplicity of notation. Our choice of entropy implies that the entropy variable $v:=\eta'(u)$ is identical to the conserved variable $u$. The variables $u$ and $v$ will be used interchangeably.
\end{remark}

\subsection{Entropy stable schemes}\label{ent_stab_schemes}
While entropy conservation is the correct notion when dealing with smooth solutions, entropy is dissipated near discontinuities in accordance to \eqref{eqn:entropyineq}. Therefore, we add entropy variable--based numerical dissipation to the numerical flux in the form
\begin{equation}\label{eqn:ESflux}
 F_\iph = F^{*,2p}_\iph - \frac{1}{2} a_\iph \Delta v_\iph
\end{equation}
where $F^{*,2p}_\iph$ is the $2p$-th order accurate entropy conservative flux described above, and $a_\iph$ is a non-negative quantity. Generally, $a_\iph$ is taken to be some appropriate approximation of $|f^\prime(u)|$. This leads us to the next lemma concerning entropy stable schemes.
\begin{lemma}[Tadmor \cite{TADMOR84}]\label{lemma:ESF}
The semi-discrete numerical scheme \eqref{eqn:fdm} with numerical flux given by \eqref{eqn:ESflux} satisfies the discrete entropy inequality
\begin{equation}
\normalfont
 \dd{\eta(u_i)}{t} + \frac{1}{h} \left(q_\iph - q_\imh\right) \leq 0
\end{equation}
where the numerical entropy flux is given by
\begin{equation*}
 q_\iph = q^*_\iph - \frac{1}{2}  \avg{v}_\iph a_\iph \Delta v_\iph.
\end{equation*}
\end{lemma}

Note that the term $\Delta v_\iph$ appearing in \eqref{eqn:ESflux} is $\mathcal{O}(|\Delta x_\iph|)$. Thus, the scheme \eqref{eqn:ESflux} is only first-order accurate, irrespective of the order of the entropy conservative flux used. To obtain a higher order scheme, we need to appropriately reconstruct the entropy variable (or equivalently, the conserved variable) at the cell interfaces. Consider the cell interface at $x_\iph$ between control volumes $I_i$ and $I_\ipo$, as shown in Figure~\ref{fig:recon}.

\begin{figure}[!ht]
\begin{center}
\includegraphics[width=0.60\textwidth]{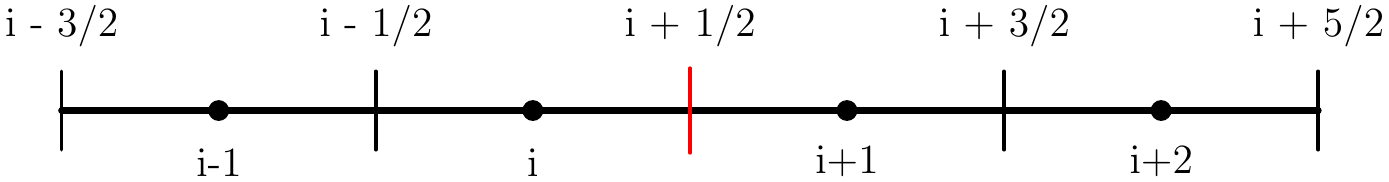}
\caption{Stencil for reconstruction.}
\label{fig:recon}
\end{center}
\end{figure}

Corresponding to this particular cell interface, let $v_i(x)$ and $v_\ipo(x)$ be polynomial reconstructions of the entropy variable in $I_i$ and $I_\ipo$ respectively. We denote the reconstructed values at the cell interface, and the difference in the reconstructed states, by
\begin{equation}
v_\iph^- = v_i(x_\iph), \qquad v_\iph^+ = v_\ipo(x_\iph), \qquad \llbracket v \rrbracket_\iph = v_\iph^+ - v_\iph^-.
\end{equation}
Replacing the original jump $\Delta v_\iph$ in \eqref{eqn:ESflux} by the reconstructed jump $\llbracket v \rrbracket_\iph$ would lead to a higher order accurate scheme. However, the entropy stability of the scheme may be lost in the process. The following result ensures that the reconstruction does not destroy entropy stability.

\begin{lemma}[Fjordholm et al.\ \cite{FMT12}]
For each interface $\iph$, if the reconstruction satisfies the sign property
\begin{equation}\label{eqn:sign_property}
{\rm sign} (\llbracket v \rrbracket_\iph) = {\rm sign} (\Delta v_\iph)
\end{equation}
then the scheme with the numerical flux 
\begin{equation}\label{eqn:ESfluxh}
 F_\iph = F^*_\iph - \frac{1}{2} a_\iph \llbracket v \rrbracket _\iph
\end{equation}
is entropy stable.
\end{lemma}
It has been shown in \cite{FMT13} that ENO interpolation satisfies the sign property, and has been tested numerically in \cite{FMT12} to give accurate results. However, to the best knowledge of the authors, other available reconstruction procedures fail to satisfy this crucial property. In the following sections we present a third-order WENO reconstruction procedure that satisfies the sign property, while maintaining the desired order of accuracy.

\section{A sign preserving WENO-3}\label{sec:weno}
The idea of WENO reconstruction is to use a suitable convex combination of all $2k-1$ polynomials used in the $k$-th order ENO reconstruction at a given interface, and obtain a $(2k-1)$-th order accurate reconstruction. We describe a procedure to choose the weights of the WENO scheme which ensures that the reconstruction is third-order accurate for smooth function, while satisfying the sign property. Consider the stencil shown in Figure~\ref{fig:recon} corresponding to the reconstructions at $x_\iph$. 

\subsubsection*{Reconstruction from the left:}
We first consider the reconstruction from the left side of the interface $x_\iph$. The two stencils for linear ENO reconstruction from the left are  
\[
 S_i^0 = \{x_i,x_\ipo\}, \quad S_i^1 = \{x_{i-1},x_i\}.
\]
The corresponding linear polynomials and their evaluations at $x_\iph$ are
\begin{equation*}
\begin{aligned}
 & p_i^{(0)}(x) = v_i \frac{(x - x_\ipo)}{(x_i - x_\ipo)} + v_\ipo \frac{(x-x_i)}{(x_\ipo - x_i)} \quad &\implies \qquad &v^{(0),-}_\iph = \frac{v_i}{2} + \frac{v_\ipo}{2},\\
& p_i^{(1)}(x) = v_\imo \frac{(x - x_i)}{(x_\imo - x_i)} + v_i \frac{(x-x_\imo)}{(x_i - x_\imo)} \quad &\implies \qquad &v^{(1),-}_\iph = -\frac{v_\imo}{2} + \frac{3 v_i}{2}.
\end{aligned}
\end{equation*}
Each of the above reconstructions at the interface $x_\iph$ are second-order accurate. Weighting each of these with non-negative values $w_{0,\iph}$ and $w_{1,\iph}$, respectively, we obtain the reconstructed value
\begin{equation}\label{eqn:vrecminus}
v^-_\iph := w_{0,\iph} \left( \frac{v_i}{2} + \frac{v_\ipo}{2} \right) + w_{1,\iph} \left( -\frac{v_\imo}{2} + \frac{3 v_i}{2} \right).
\end{equation}
The weights must be chosen such that third-order accuracy is achieved, so we require that
\begin{align*}
v(x_\iph) + \mathcal{O}(h^3) = v^-_\iph &= w_{0,\iph} \left( \frac{v_i}{2} + \frac{v_\ipo}{2} \right) + w_{1,\iph} \left( -\frac{v_\imo}{2} + \frac{3 v_i}{2} \right) \\*
              &= w_{0,\iph} \left(  v(x_\iph) + \frac{1}{8} v^{\prime \prime}(x_\iph)  h^2 + \mathcal{O}(h^3)\right) \\*
              &\quad + w_{1,\iph} \left(  v(x_\iph) - \frac{3}{8} v^{\prime \prime}(x_\iph)  h^2 + \mathcal{O}(h^3)\right).
\end{align*}
Comparing the left and right hand sides, we obtain the following constraints on the weights:
\begin{subequations}\label{eqn:left}
\begin{align}
w_{0,\iph} + w_{1,\iph} &= 1, \label{eqn:leftA}\\
C_1 := \frac{w_{0,\iph}}{8} - \frac{3w_{1,\iph}}{8} &= \mathcal{O}(h).\label{eqn:leftB}
\end{align}
\end{subequations}

\subsubsection*{Reconstruction from the right:}
We now consider the reconstruction from the right at the interface $x_\iph$, which requires the stencils
\[
 \widetilde{S}_0 = \{x_\ipo,x_{i+2}\}, \quad \widetilde{S}_1 = \{x_i,x_\ipo\}.
\]
The corresponding polynomial and their evaluations at $x_\iph$ are
\begin{equation*}
\begin{aligned}
 &\widetilde{p}^{(0)}(x) = v_\ipo \frac{(x - x_{i+2})}{(x_\ipo - x_{i+2})} + v_{i+2} \frac{(x-x_\ipo)}{(x_{i+2}- x_\ipo)} \quad &\implies \qquad & v^{(0),+}_\iph = \frac{3v_\ipo}{2} - \frac{v_{i+2}}{2}, \\
 &\widetilde{p}^{(1)}(x) =v_i \frac{(x - x_\ipo)}{(x_i - x_\ipo)} + v_\ipo \frac{(x-x_i)}{(x_\ipo - x_i)} \quad &\implies \qquad & v^{(1),+}_\iph = \frac{v_i}{2} + \frac{v_\ipo}{2}.
\end{aligned}
\end{equation*}
Let the weights in this case be denoted by $\wt_{0,\iph}$ and $\wt_{1,\iph}$. As before, we set
\begin{equation}\label{eqn:vrecplus}
v^+_\iph := \widetilde{w}_{0,\iph} \left( -\frac{v_{i+2}}{2} + \frac{3 v_\ipo}{2} \right) + \widetilde{w}_{1,\iph} \left( \frac{v_i}{2} + \frac{v_\ipo}{2} \right),
\end{equation}
and we require
\begin{align*}
v(x_\iph) + \mathcal{O}(h^3) &= \widetilde{w}_{0,\iph} \left(  v(x_\iph) - \frac{3}{8} v^{\prime \prime}(x_\iph)  h^2 + \mathcal{O}(h^3)\right) \\
&\quad + \widetilde{w}_{1,\iph} \left( v(x_\iph)  + \frac{1}{8} v^{\prime \prime}(x_\iph)  h^2 + \mathcal{O}(h^3)\right).
\end{align*}
This enforces the following constraints on the weights:
\begin{subequations}\label{eqn:right}
\begin{align}
\widetilde{w}_{0,\iph}+ \widetilde{w}_{1,\iph} &= 1, \label{eqn:rightA}\\
C_2 := -\frac{3\widetilde{w}_{0,\iph}}{8} + \frac{\widetilde{w}_{1,\iph}}{8} &= \mathcal{O}(h).\label{eqn:rightB}
\end{align}
\end{subequations}

The weights $w_{1,\iph}, w_{1,\iph}, \widetilde{w}_{1,\iph}, \widetilde{w}_{1,\iph}$ must be chosen in accordance to \eqref{eqn:left} and \eqref{eqn:right} to ensure that the desired consistency and accuracy of the reconstruction is achieved. For the remainder of this paper we drop the $\iph$ subscript in the weights wherever it is clear that we are referring to the interface $x_\iph$.

\subsection{Properties}
We list the various crucial properties that the reconstruction needs to possess, including the sign-property.

\subsubsection{Consistency}
Using \eqref{eqn:leftB} and \eqref{eqn:rightB}, we rewrite the weights as
\[
w_0 = \frac{3}{4} + 2C_1, \quad w_1 = \frac{1}{4} - 2C_1, \quad \wt_0 = \frac{1}{4} - 2C_2, \quad \wt_1 = \frac{3}{4} + 2C_2.
\]
To ensure that the weights are non-negative and that \eqref{eqn:leftA} and \eqref{eqn:rightA} are satisfied, we require the following consistency condition:
\begin{equation}\label{eqn:cons_a}\tag{P1}
0 \leq w_0,w_1,\wt_0,\wt_1 \leq 1,
\end{equation}
or equivalently,
\begin{equation}\label{eqn:cons_b}\tag{P1'}
-\frac{3}{8} \leq C_1,C_2  \leq \frac{1}{8}.
\end{equation}

\subsubsection{Sign property}
The jump in the reconstructed variables can be written as
\begin{align*}
\llbracket v \rrbracket_\iph &=  \left[ -\frac{ \widetilde{w}_0}{2} \thl_\ipo  + \frac{\left(\widetilde{w}_0 +  w_1 \right)}{2}  -\frac{ w_1}{2} \thr_i \right] \Delta v_\iph \\
&=\frac{1}{2}\left[ \widetilde{w}_0( 1- \thl_\ipo)  + w_1(1 - \thr_i) \right] \Delta v_\iph
\end{align*}
where
\[
\thl_i = \frac{\Delta v_\iph}{\Delta v_\imh}, \qquad \thr_i = \frac{1}{\thl_i} = \frac{\Delta v_\imh}{\Delta v_\iph}.
\]
Thus, the following is an \emph{equivalent} formulation of the sign property \eqref{eqn:sign_property}, whenever $\Delta v_\iph \neq 0$:
\begin{equation}\label{eqn:signprop}\tag{P2}
\left[ \widetilde{w}_0( 1- \thl_\ipo)  + w_1(1 - \thr_i) \right] \geq 0.
\end{equation}

\subsubsection{Negation symmetry}
By  {\em negation symmetry}, we mean that the weights are not biased towards positive or negative solution values. In other words, the weights should remain unchanged under the transformation $v \mapsto -v$. The jumps accordingly transform as
\[
\Delta v_\jph \mapsto -\Delta v_\jph \qquad \forall\ j \in \mathbb{Z}.
\]
However, the jump ratios $\thl_j$ or $\thr_j$ remain unchanged. 
A sufficient condition to enforce negation symmetry is to choose $C_1,C_2$ as functions of $\thr_i,\thl_\ipo$:
\begin{equation}\label{eqn:negsym}\tag{P3}
 C_1 = C_1 (\thr_i,\thl_\ipo), \qquad C_2 = C_2 (\thr_i,\thl_\ipo).
\end{equation}

\subsubsection{Mirror property}
If we mirror the solution about the interface $x_\iph$, we would like to ensure that the weights also get mirrored about $x_\iph$. The mirroring transforms the various difference ratios as follows:
\[
\thl_\ipo  \mapsto \thr_i, \qquad  \thr_i  \mapsto \thl_\ipo.
\]
It is straightforward to see that the weights must transform as
\[
 w_0 \mapsto \widetilde{w}_1, \qquad w_1 \mapsto \widetilde{w}_0, \qquad \widetilde{w}_0 \mapsto w_1, \qquad \widetilde{w}_1 \mapsto w_0.
\]
\emph{Assuming} that the negation symmetry \eqref{eqn:negsym} holds, the above is true \emph{if and only if}
\begin{equation}\label{eqn:mirsym}\tag{P4}
C_1 (a,b) = C_2 (b,a) \qquad \forall\ a,b\in\mathbb{Z}.
\end{equation}

\subsubsection{Inner jump condition}
In addition to the sign property, we would like the reconstructed variables to satisfy the \textit{inner jump condition} in each cell $i$: 
\[
\text{sign}(v^-_\iph - v^+_\imh) =  \text{sign}(\Delta v_\iph) = \text{sign}(\Delta v_\imh)
\]
whenever the second equality holds. This property ensures that the monotonicity of the solution is preserved. The second-order ENO reconstruction satisfies this property, while it need not hold true for higher order ENO. Recalling the definitions \eqref{eqn:vrecminus}, \eqref{eqn:vrecplus} of $v^-_\iph$ and $v^+_\imh$, we obtain
 \begin{align*}
 v^-_\iph - v^+_\imh&= w_{0,\iph} \left( \frac{v_i}{2} + \frac{v_\ipo}{2} \right) + w_{1,\iph} \left( -\frac{v_\imo}{2} + \frac{3 v_i}{2} \right) \\
 &\quad- \widetilde{w}_{0,\imh} \left( -\frac{v_\ipo}{2} + \frac{3 v_i}{2} \right) - \widetilde{w}_{1,\imh} \left( \frac{v_i}{2} + \frac{v_\imo}{2} \right) \\
&= \Delta v_\iph \frac{w_{0,\iph} + \widetilde{w}_{0,\imh}}{2} + \Delta v_\imh \frac{w_{1,\iph} + \widetilde{w}_{1,\imh}}{2}.
 \end{align*} 
By the assumption \eqref{eqn:cons_a} of non-negativity of the weights, the coefficients of $\Delta v_\imh$ and $\Delta v_\iph$ in the above expression are non-negative, and hence the inner jump condition is \emph{automatically satisfied}.

\subsubsection{Accuracy}\label{sec:accuracy}
In general, conditions \eqref{eqn:leftB} and \eqref{eqn:rightB} require $C_1$ and $C_2$ to be $\mathcal{O}(h)$ for smooth solutions. However, this condition can be relaxed in scenarios in which $v^{\prime \prime}(\hat{x}) = 0$, for some $\hat{x}$ such that $|\hat{x} - x_\iph| = \mathcal{O}(h)$. Figure \ref{fig:specialcase} depicts a few situations in which this can happen.
\begin{figure}[h]
\begin{center}
\subfigure[]{\includegraphics[width=0.30\textwidth]{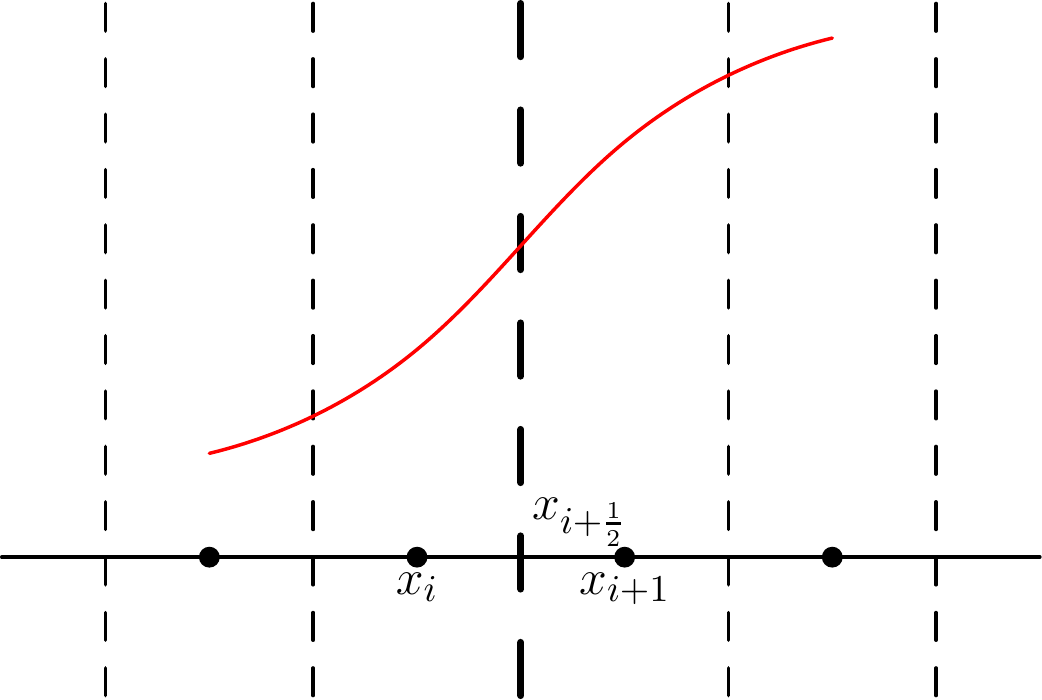}}
\subfigure[]{\includegraphics[width=0.30\textwidth]{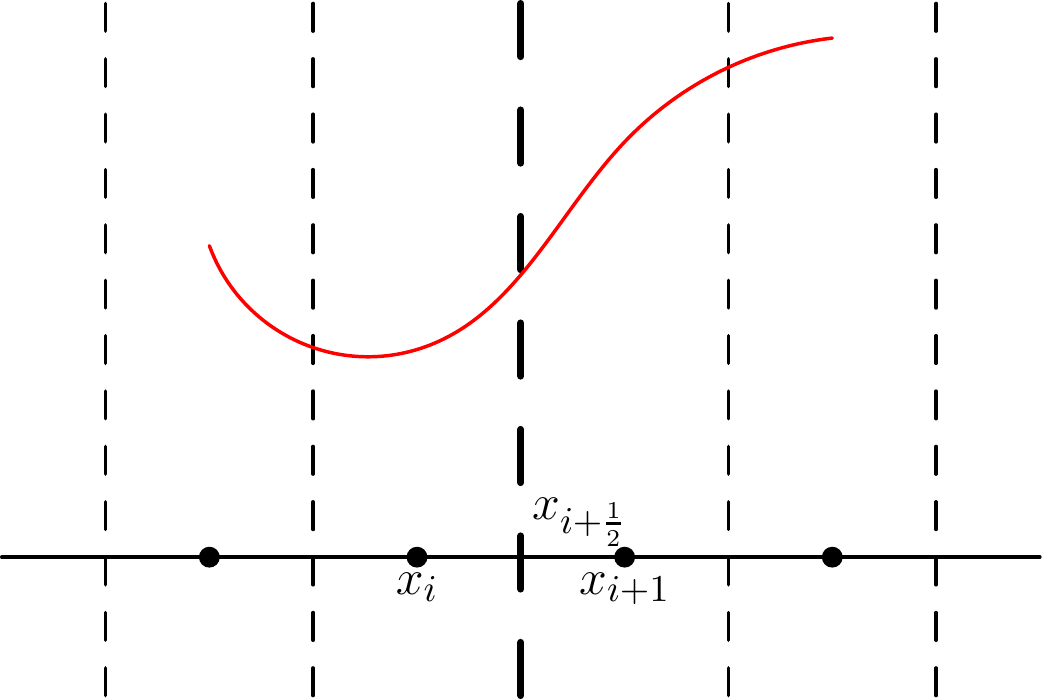}}
\subfigure[]{\includegraphics[width=0.30\textwidth]{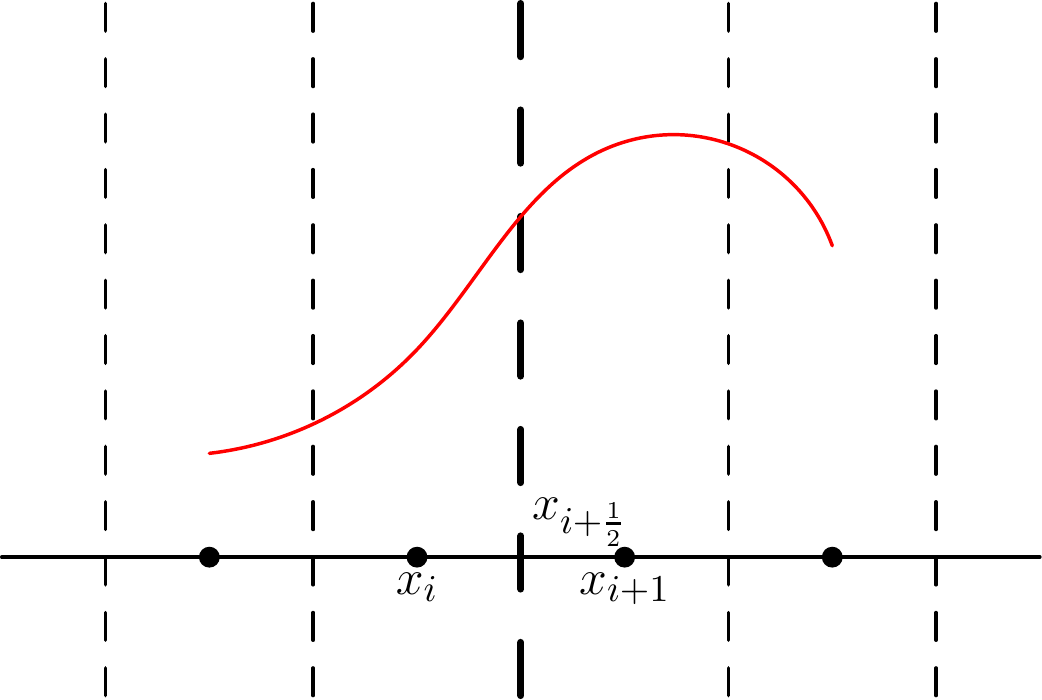}}\\
\subfigure[]{\includegraphics[width=0.30\textwidth]{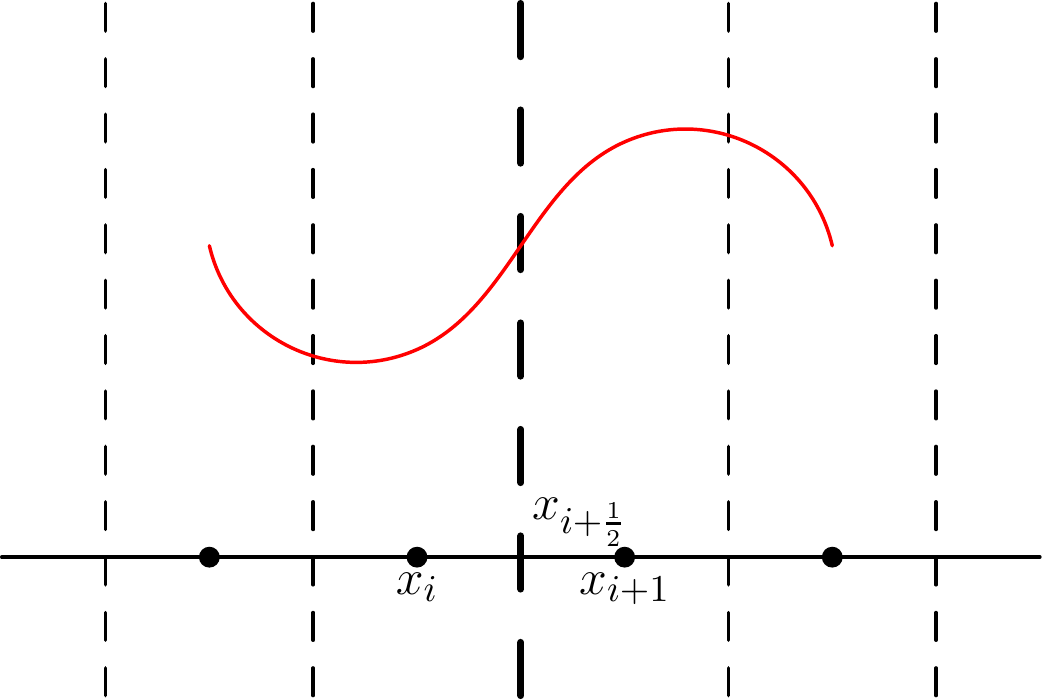}}
\subfigure[]{\includegraphics[width=0.30\textwidth]{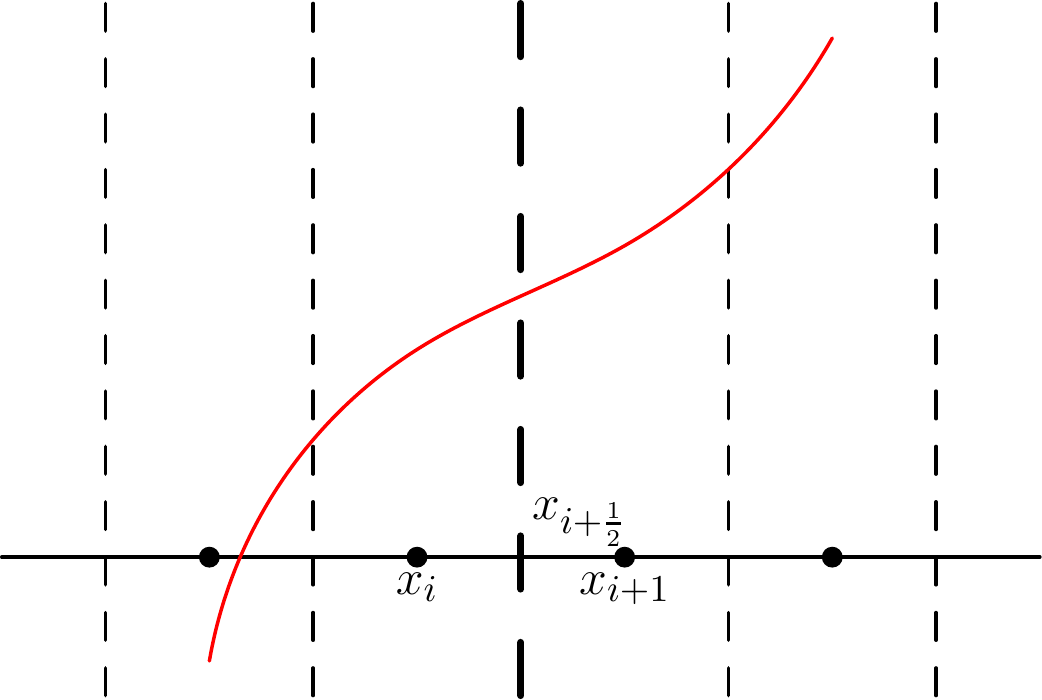}}
\subfigure[]{\includegraphics[width=0.30\textwidth]{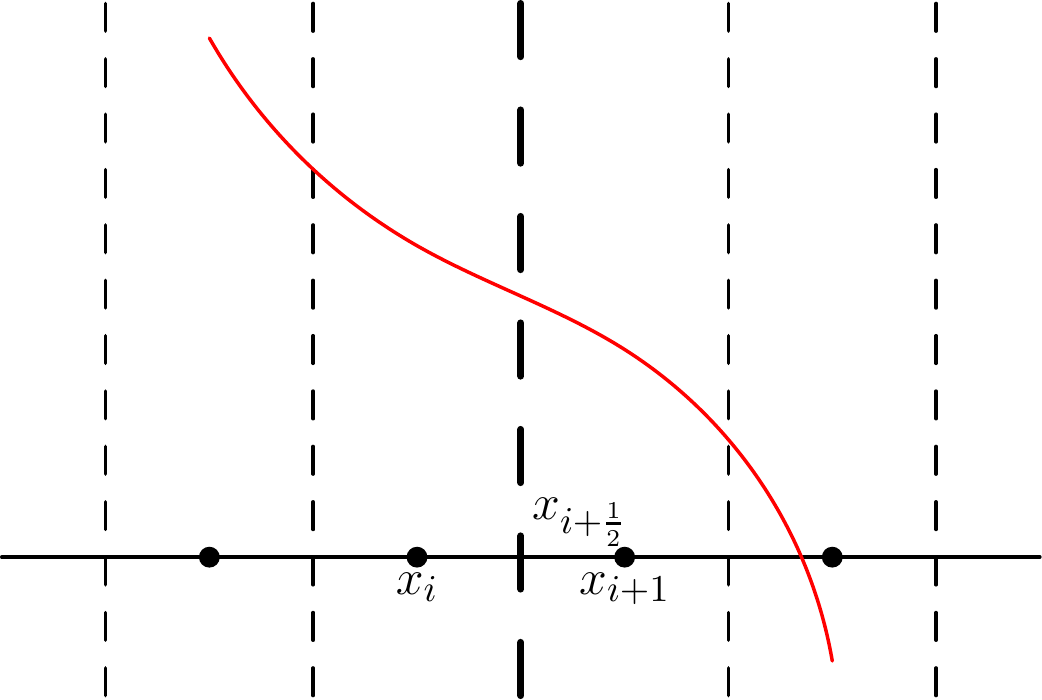}}
\caption{Special cases when $v^{\prime \prime}(\hat{x}) = 0$.}
\label{fig:specialcase}
\end{center}
\end{figure}
Assuming sufficient regularity on the solution, the following is true in these special scenarios:
\[
 v^{\prime \prime}(x_\iph) = v^{\prime \prime}(\hat{x}) + (x_\iph - \hat{x}) v^{\prime \prime \prime}(\hat{x}) + \mathcal{O}(h^2) = \mathcal{O}(h).
\]
This in turn implies that each of the linear polynomials used for reconstruction gives a third-order accurate approximation of the solution at $x_\iph$. Thus, \eqref{eqn:leftB} and \eqref{eqn:rightB} become redundant. In other words, the reconstruction is third-order accurate provided
\begin{equation}\label{eqn:variableorder}\tag{P5}
C_1, C_2 =  \begin{cases} \mathcal{O}(h), \hspace{3.2cm} \text{ in GC}  \\
                                       \text{no order restriction}, \hspace{1cm} \text{ in SC} \end{cases}
\end{equation}
where we use the abbreviations GC and SC to denote general cases and special cases respectively.

\subsection{The feasible region}\label{sec:feas_region}
In order to choose weights satisfying \eqref{eqn:cons_a}--\eqref{eqn:variableorder}, we first analyse how the weights behave under the above constraints. We will look at six different scenarios, depending on the values of $\thr_i, \thl_\ipo$. In each scenario we will try to determine the {\em feasible region}, which corresponds the region where the weights satisfy \eqref{eqn:cons_a} and \eqref{eqn:signprop}. The remaining properties will be considered in Section \ref{sec:sp_weno} while trying to construct explicit weights. We define 
\[
\psi^+_\iph := \frac{( 1- \thl_\ipo)}{(1 - \thr_i)}, \qquad \psi^-_\iph := \frac{1}{\psi^+_\iph}.
\]
The $\iph$ subscript will be dropped whenever it is obvious that we are referring to the interface $\iph$. We also introduce the notation
\[
 \mathcal{L}:= \begin{cases} 
                       \frac{C_1}{\frac{1}{8}(1+\psi^+)} +  \frac{C_2}{\frac{1}{8}(1+\psi^-)}, &\mbox{if } \psi^+ \neq -1 \\
                       C_1 - C_2 + 1, &\mbox{if } \psi^+ = \psi^- = -1.
                      \end{cases} 
\]
Furthermore, we denote the open box $\left(-\frac{3}{8},\frac{1}{8} \right) \times \left(-\frac{3}{8},\frac{1}{8} \right)$ by $\mathcal{B}$. Recall that the consistency constraint \eqref{eqn:leftB} requires that $(C_1,C_2)\in\overline{\mathcal{B}}$.

\subsubsection{Case 1: $\thr_i,\thl_\ipo > 1$}
The qualitative nature of the (smooth) solution for this case is indicated in Figure \ref{fig:case1}. The solution is clearly not strictly convex or concave in the stencil under consideration, even if the soultion is more oscillatory than that shown in Figure~\ref{fig:case1}. Thus, we are in the SC regime, which implies that no order of accuracy restrictions must be imposed on $C_1,C_2$. To ensure that \eqref{eqn:signprop} holds, we need 
\[
 \wt_0 = w_1 = 0 \qquad \iff \qquad C_1 = C_2 = \frac{1}{8}.
\]
Note that this leads to precisely the ENO-2 stencil selection, which is suitable for discontinuous solutions as well.
\begin{figure}
\begin{center}
\subfigure[]{\includegraphics[width=0.35\textwidth]{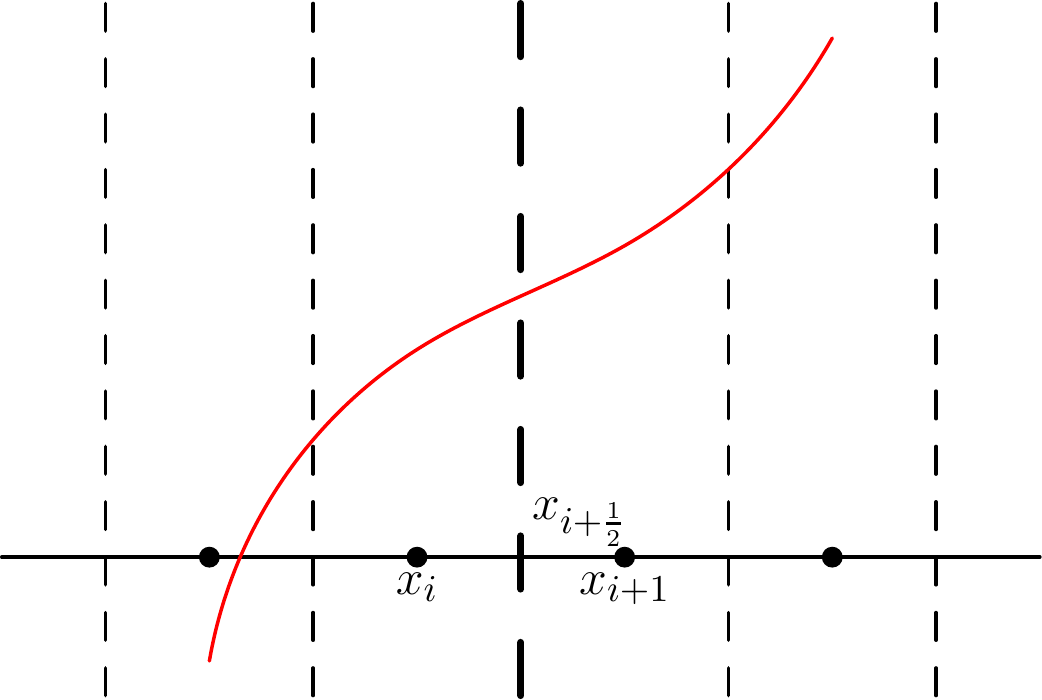}} \hspace{2ex}
\subfigure[]{\includegraphics[width=0.35\textwidth]{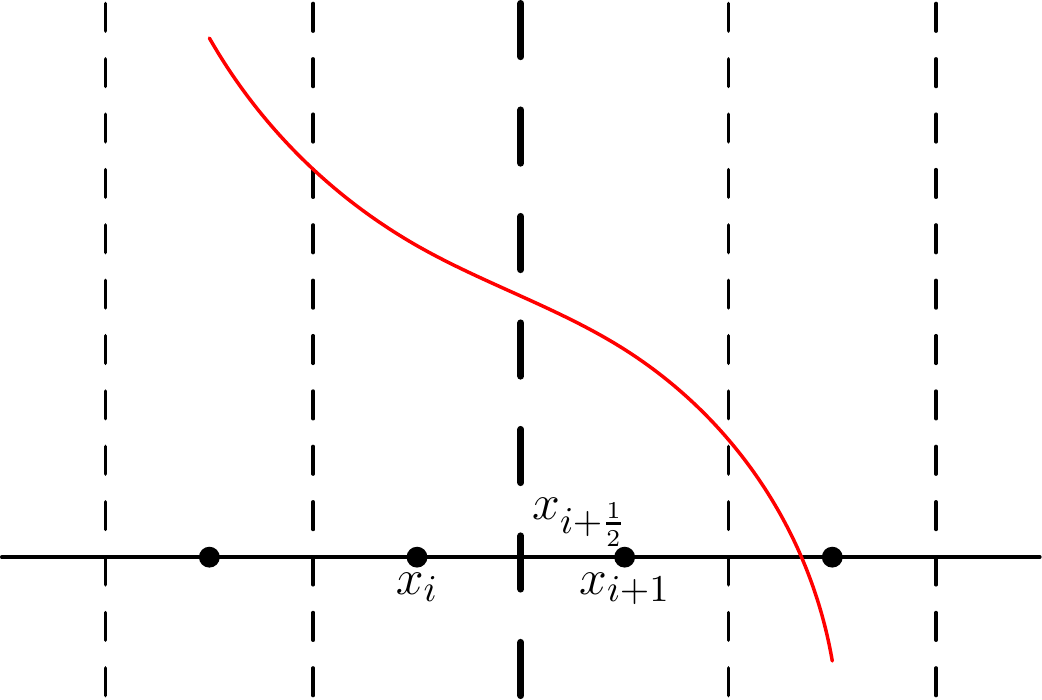}}\\
\caption{Possible scenarios for Case 1.}
\label{fig:case1}
\end{center}
\end{figure}

\subsubsection{Case 2: $\thr_i< 1, \ \thl_\ipo > 1$} 
In this case we have
\[
\psi^+ < 0, \qquad 1+\psi^+ < 1.
\]
Case 2 falls into the GC regime. Thus, we must choose $C_1$ and $C_2$ carefully so as not to violate the accuracy condition. The sign property \eqref{eqn:signprop} will hold if
\[
 w_1 \geq -\wt_0 \psi^+ \quad \iff \quad \left(\frac{1}{4} - 2C_1\right) \geq - \left(\frac{1}{4} - 2C_2 \right) \psi^+ \quad \iff \quad C_1 + \psi^+ C_2 \leq \frac{1}{8}(1 + \psi^+).
\]
Thus, we have the following constraints on $C_1$, $C_2$:
\begin{align*}
-\frac{3}{8} < C_1,C_2 &< \frac{1}{8}, \\                                       
\mathcal{L} &\leq 1  \hspace{1cm} \text{if } -1 \leq \psi^+<0 ,\\ 
\mathcal{L} &\geq 1\hspace{1cm} \text{if }  \psi^+<-1.
\end{align*}
The feasible region for $C_1$, $C_2$ is shown in Figure \ref{fig:case23fr} (in dark grey).
\begin{figure}
\begin{center}
\subfigure[$\psi^+ = -0.5 \in (-1,0)$]{\includegraphics[width=0.49\textwidth]{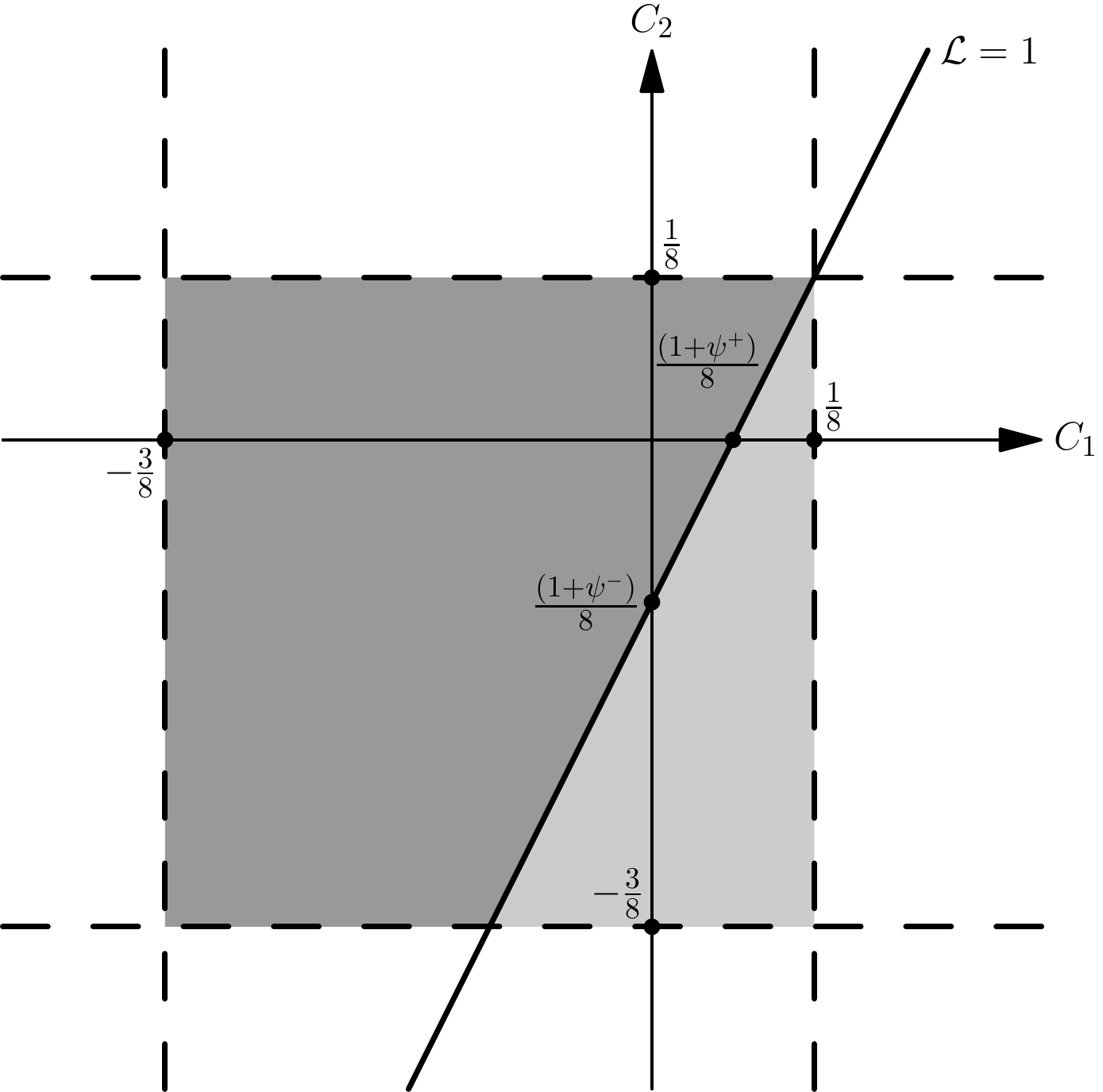}}
\subfigure[$\psi^+ = -1.5 \in (-\infty,-1)$]{\includegraphics[width=0.49\textwidth]{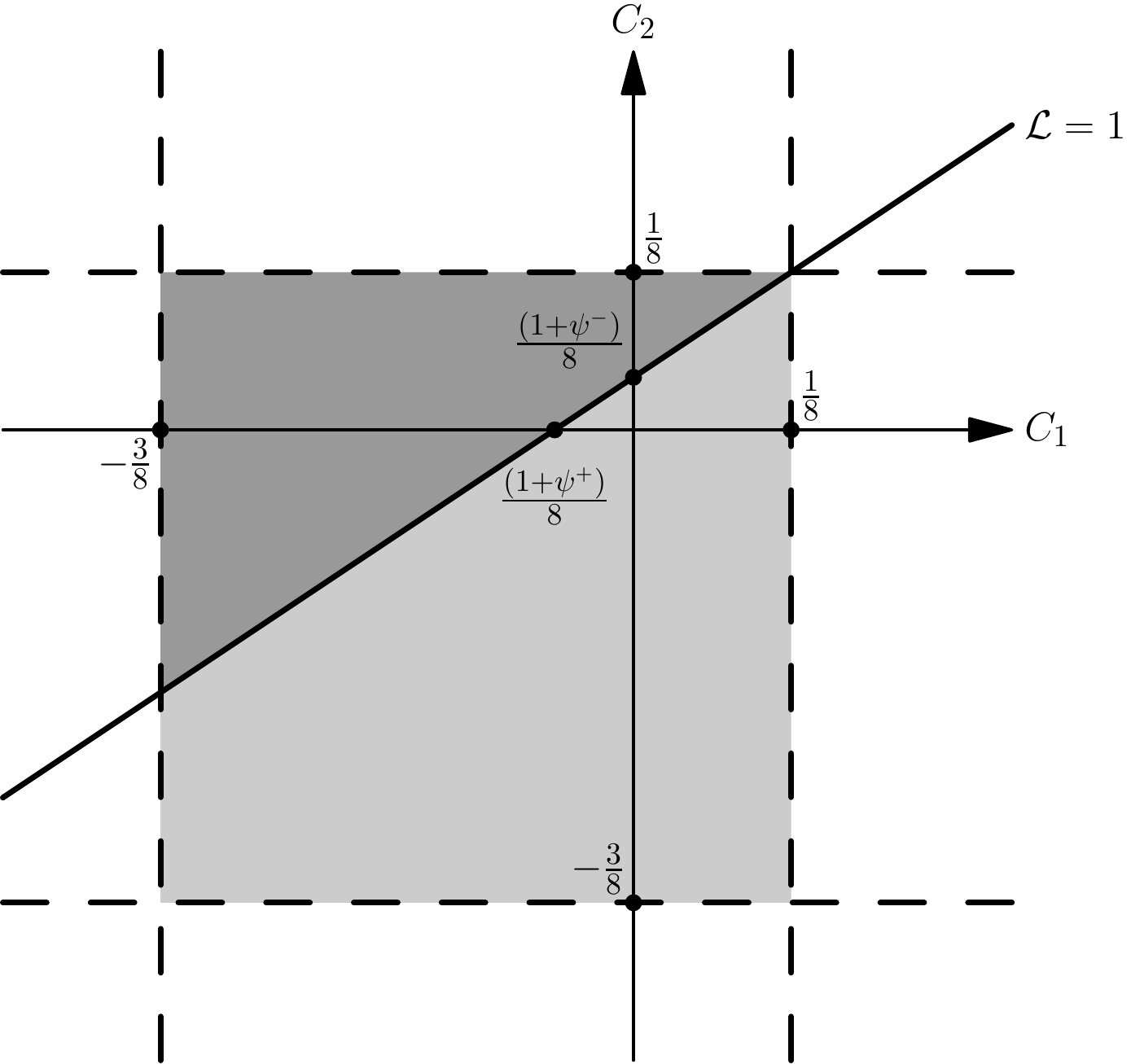}}\\
\caption{Feasible region for Case 2 (dark grey) and Case 3 (light grey).}
\label{fig:case23fr}
\end{center}
\end{figure}

\subsubsection{Case 3: $\thr_i> 1,\ \thl_\ipo < 1$} 
Similiar to case 2, we have 
\[
\psi^+ < 0, \qquad 1+\psi^+ < 1
\]
with the solutions falling into the GC regime in general. The sign property holds if
\[
 \wt_0 \geq -w_1 \psi^- \quad \iff \quad \left(\frac{1}{4} - 2C_2 \right) \geq - \left(\frac{1}{4} - 2C_1 \right) \psi^- \quad \iff \quad C_1 + \psi^+ C_2 \geq \frac{1}{8}(1 + \psi^+).
\]
Comparing the last equivalent condition above with that observed for case 2, we see that the inequality has been flipped. Thus, we have the following  constraints on $C_1, C_2$:
\begin{align*}
-\frac{3}{8} < C_1,C_2 &< \frac{1}{8}, \\*                                       
\mathcal{L} &\geq 1  \hspace{1cm} \text{if } -1 \leq \psi^+<0,\\*
\mathcal{L} &\leq 1\hspace{1cm} \text{if }  \psi^+<-1.
\end{align*}
The feasible region for $C_1,C_2$ is shown in Figure~\ref{fig:case23fr} (in light grey).

\begin{remark}\label{rem:case23a}
Any point in $ \{(C_1,C_2) : \mathcal{L} = 1\} \cap \mathcal{B}$, satisfies the constraints in cases 2 and 3. This fact will be exploited in constructing explicit weights.
\end{remark} 

\subsubsection{Case 4: $\thl_\ipo = 1$} 
In this case the solution either has a linear region or is oscillatory without being strictly convex or concave. Thus, this case falls in the SC regime. In order to satisfy \eqref{eqn:signprop}, we require $w_1$ to have the same sign as $(1-\thr_i)$.
If $\thr_i \leq 1$, then $(C_1,C_2)$ can be chosen as any point in $\overline{\mathcal{B}}$. However, if $\thr_i > 1$, then to satisfy \eqref{eqn:cons_a} and \eqref{eqn:signprop}, we must take $C_1 =\frac{1}{8}$, while $C_2$ can be any value in $\left[-\frac{3}{8},\frac{1}{8} \right]$. This would lead to $w_0 = 1, w_1 = 0$ when  $\thr_i > 1$, which is identical to the ENO-2 stencil selection. 

\subsubsection{Case 5: $\thr_i = 1$} 
This case is similar to case 4, with the values of $\thr_i$ and $\thl_\ipo$ interchanged. If $\thl_\ipo \leq 1$, then $(C_1,C_2)$ can be chosen as any point in $\overline{\mathcal{B}}$. If $\thl_\ipo > 1$, then we must take $C_2 =\frac{1}{8}$ while $C_1$ can be any value in $\left[-\frac{3}{8},\frac{1}{8} \right]$.

\subsubsection{Case 6: $\thr_i, \thl_\ipo < 1$} 
In this final case, we have
\[
\psi^+ > 0,  \qquad 1+\psi^+ > 1.
\]
Note that this was true in case 1 as well. By an argument similar to the one made in case 1, we can show that case 6 falls into the SC regime (see Figure \ref{fig:case6}). Furthermore, the sign property is satisfied as long as the consistency condition \eqref{eqn:cons_a} holds true.
\begin{figure}[!htbp]
\begin{center}
\subfigure[]{\includegraphics[width=0.32\textwidth]{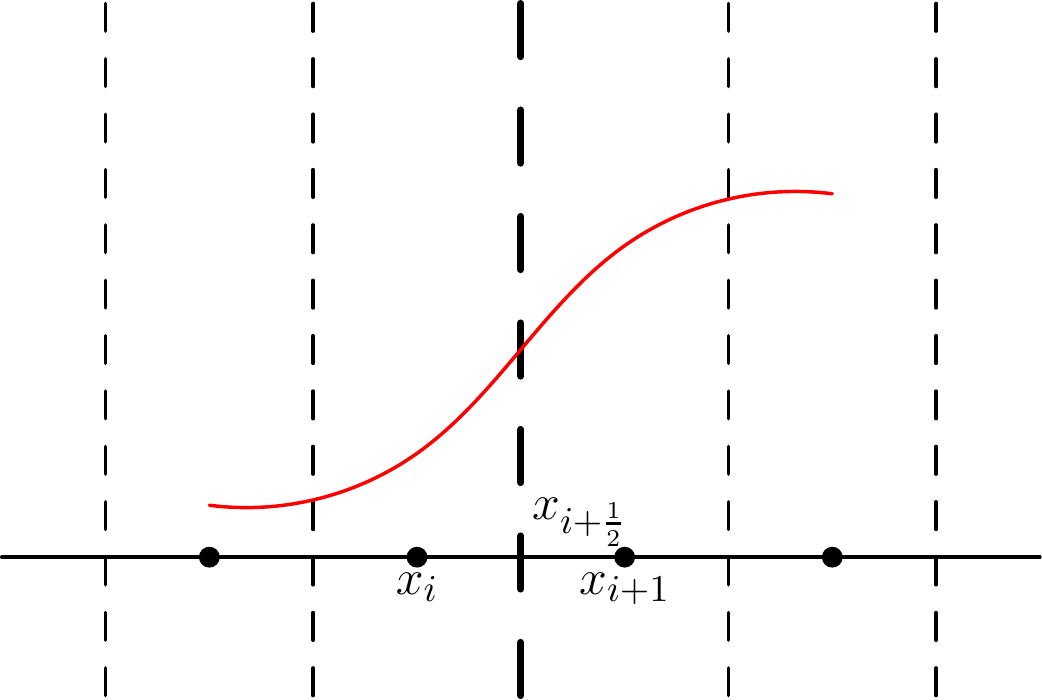}}
\subfigure[]{\includegraphics[width=0.32\textwidth]{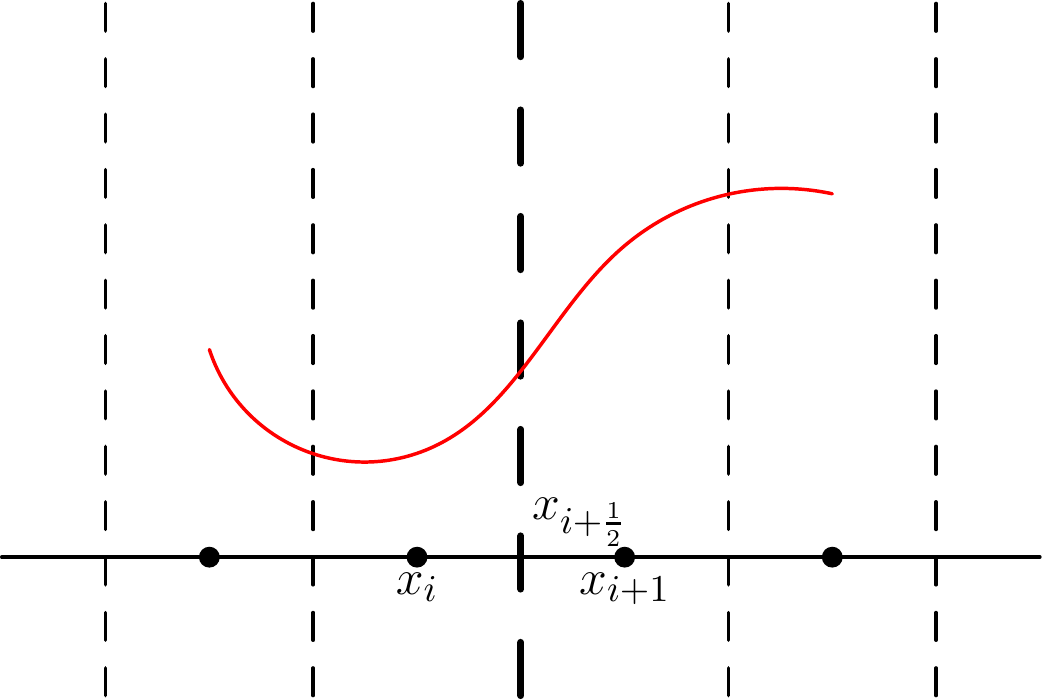}}
\subfigure[]{\includegraphics[width=0.32\textwidth]{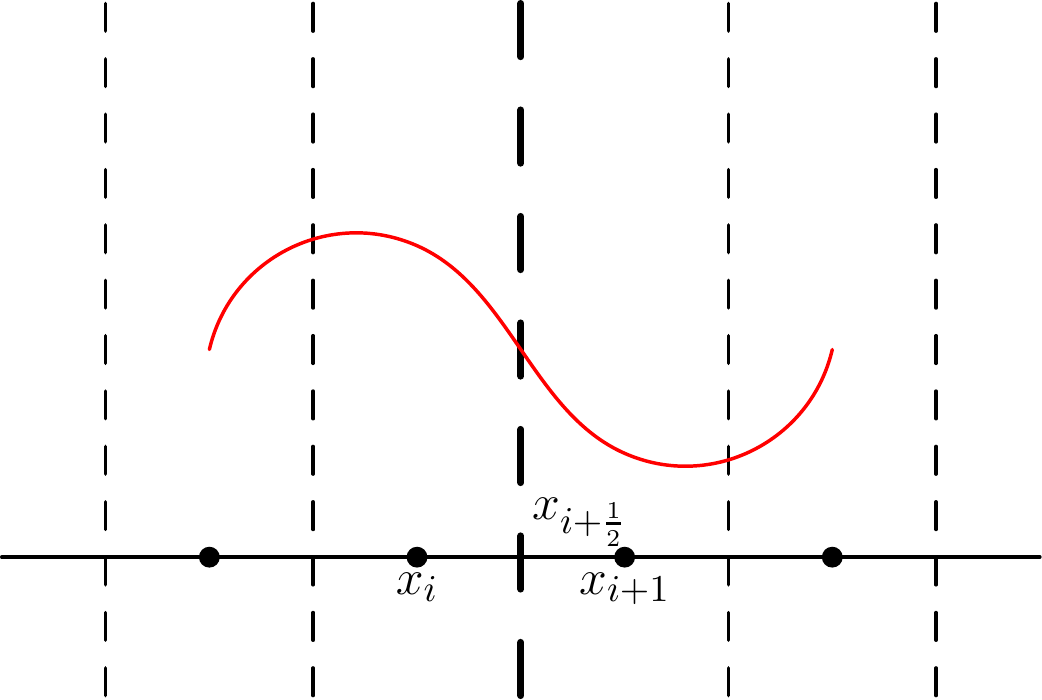}}
\caption{Possible scenarios for Case 6.}
\label{fig:case6}
\end{center}
\end{figure}

\section{Explicit weights: SP-WENO}\label{sec:sp_weno}
We now make an explicit choice for the weights, based on the case by case analysis in the previous section. Recall that from an accuracy point of view, the optimal choice of weights is $(w_0,w_1) = (\unitfrac{3}{4},\unitfrac{1}{4})$ and $(\widetilde{w}_0,\widetilde{w}_1) = (\unitfrac{1}{4},\unitfrac{3}{4})$, or equivalently, $(C_0,C_1) = (0,0)$. However, the point $(C_0,C_1) = (0,0)$ in many cases does not lie in the feasible region (cf.\ Section \ref{sec:feas_region}).

For cases 2 and 3, we choose $(C_0,C_1)$ to be the point in $ \{\mathcal{L} = 1\} \cap \mathcal{B}$, by virtue of Remark \ref{rem:case23a}. Furthermore, $C_0$ and $C_1$ must both be of order $\mathcal{O}(h)$ in these two cases, in order to satisfy \eqref{eqn:variableorder}. Thus, we choose $(C_0,C_1)$ to be the point on $\mathcal{L} = 1$ closest to the origin, as measured in the Euclidean norm:
\begin{equation}\label{eqn:opt1case23}
C_1(\thr_i,\thl_\ipo) = \begin{cases}\frac{1}{8} \left(\frac{f^+}{(f^+)^2 + (f^-)^2}\right) &\text{if } \psi^+ \neq -1 \\
0 & \text{otherwise,}\end{cases}
\qquad
C_2(\thr_i,\thl_\ipo) = C_1(\thl_\ipo,\thr_i),
\end{equation}
where we have defined
\[
f^+(\thr_i,\thl_\ipo) := \begin{cases} \frac{1}{1 + \psi^+} &\text{if } \thr_i \neq 1, \psi^+ \neq -1 \\
1 & \text{otherwise,} \end{cases}
\qquad
f^-(\thr_i,\thl_\ipo) := f^+(\thl_\ipo,\thr_i).
\]
For smooth functions $v$ we have $(1+ \psi^+), (1-\psi^-) = \mathcal{O}(h)$, so $C_1,C_2 = \mathcal{O}(h)$ for cases 2 and 3, and hence \eqref{eqn:variableorder} is clearly satisfied. 

For cases 1 and 6, the point on the line $\mathcal{L}=1$ closest to the origin need not lie in the feasible region. Furthermore, $\psi^+$ and $\psi^-$ need not be defined for cases 4 and 5, thus the line $\mathcal{L}=1$ is not defined. Note that for these remainder cases, there is no order restriction on $(C_0,C_1)$. Going through a case-by-case analysis, we propose the following extension of \eqref{eqn:opt1case23}:
\begin{equation}\label{eqn:opt1}
C_1(\thr_i,\thl_\ipo) = \begin{cases}
	\frac{1}{8} \left(\frac{f^+}{(f^+)^2 + (f^-)^2}\right) &\text{if } \thr_i \neq 1,\ \psi^+ < 0,\ \psi^+ \neq -1 \\
	0 &\text{if } \thr_i \neq 1, \psi^+ = -1 \\
	-\frac{3}{8} & \text{if } \thr_i = 1 \text{ or } \psi^+ \geq 0,\ |\thr_i|\leq 1 \\
	\frac{1}{8} & \text{if } \psi^+ \geq 0,\ |\thr_i|> 1
\end{cases}
\end{equation}
and $C_2(\thr_i,\thl_\ipo) = C_1(\thl_\ipo,\thr_i),$ as before.

By virtue of lying in the feasible region, the above choices of $C_1,C_2$ automatically satisfy the consistency and sign properties \eqref{eqn:cons_a}, \eqref{eqn:signprop}. By definition they satisfy the negation and mirror symmetry properties \eqref{eqn:negsym} and \eqref{eqn:mirsym}, and through a case-by-case analysis, it can be seen that the weights also satisfy the accuracy condition \eqref{eqn:variableorder}.

We refer to the third-order WENO-type reconstruction method as \emph{SP-WENO}.

\begin{remark}
In the above discussion, we have assumed that $\Delta v_\iph \neq 0$. If $\Delta v_\iph = 0$, then the weights are chosen to be $w_1 = \wt_0 = 0$, leading to $\llbracket v \rrbracket_\iph = 0$.
\end{remark}
\begin{remark}\label{rem:case23b}
It can be shown that choosing $C_1,C_2$ to be a point on the line $\mathcal{L} = 1$ ensures that the reconstructed states are equal, i.e.\ $v_\iph^- = v_\iph^+$. Thus, $\llbracket v \rrbracket_\iph = 0$ for cases 2 and 3.
\end{remark}
\begin{remark}\label{rem:tvdprop}
During the making of this paper, we were able to find several other WENO-3 weights satisfying the above mentioned properties. The weights described by \eqref{eqn:opt1} gave the best numerical results among all the possible options considered, especially near discontinuous solutions. We were also able to construct weights which ensured the reconstruction to be TVD. However, the reconstruction suffered from loss of accuracy near smooth extrema, leading to overall second-order accuracy. This is to be expected; cf.\ Osher, Chakravarthy \cite{OSHER84b}. Furthermore, the search for TVD property in the set-up of TeCNO schemes would be futile, as the high-order entropy conservative fluxes used do not lead to a TVD scheme.
\end{remark}

\subsection{Stability estimates}
We now show that it is possible to estimate the reconstructed jumps in terms of the original jumps. In order to do so, we first express the reconstructed jumps in terms of the original jumps.

In case 1, we have $\thr_i,\thl_\ipo > 1$ and $w_1 = \wt_0 = 0$. Thus, the reconstructed states 
\[
v_\iph^- = v_\iph^+ = \frac{1}{2}(v_i + v_\ipo)
\]
have zero jump. This is also true for cases 2--3, by virtue of Remark \ref{rem:case23b}. Proceeding in a similar manner for cases 4--6, we find that the jump in reconstructed states is
\begin{equation*}
 \llbracket v \rrbracket_\iph =  \begin{cases}
                                                0 &  \text{if } \hspace{1cm}  \begin{rcases}
                                                 \thr_i > 1 \text{ and } \thl_\ipo > 1 \hspace{2cm} \text{ (case 1)} \\
                                                \thr_i < 1 \text{ and } \thl_\ipo > 1 \hspace{2cm} \text{ (case 2)}\\
                                                \thr_i > 1 \text{ and } \thl_\ipo < 1 \hspace{2cm} \text{ (case 3)}\\
                                                  |\thr_i| >1 \text{ and } \thl_\ipo = 1 \hspace{1.8cm} \text{ (case 4)}\\
                                                  \thr_i = 1 \text{ and } |\thl_\ipo| >  1 \hspace{1.8cm} \text{ (case 5)}\\
                                                   \thr_i <-1 \text{ and } \thl_\ipo < -1 \hspace{1.4cm} \text{ (case 6)}\end{rcases} \Omega_0
                                                   \\
                                                   \\
                                                  \frac{1}{2}(\Delta v_\iph - \Delta v_\imh) &\text{if } \hspace{1cm}  \begin{rcases}
                                                   |\thr_i| \leq 1 \text{ and } \thl_\ipo = 1 \hspace{1.8cm} \text{ (case 4)}\\
                                                -1 \leq \thr_i < 1 \text{ and } \thl_\ipo <-1 \hspace{0.8cm} \text{ (case 6)}\end{rcases} \Omega_1
                                                   \\
                                                   \\
                                                 \frac{1}{2}(\Delta v_\iph - \Delta v_{i + \frac{3}{2}}) &\text{if } \hspace{1cm}  \begin{rcases}
                                                  \thr_i = 1 \text{ and } |\thl_\ipo| \leq 1 \hspace{1.8cm} \text{ (case 5)}\\
                                               \thr_i < -1 \text{ and } -1 \leq \thl_\ipo <1 \hspace{0.6cm} \text{ (case 6)}\end{rcases} \Omega_2
                                                   \\
                                                   \\
                                               \Delta v_\iph - \frac{1}{2}(\Delta v_\imh + \Delta v_{i + \frac{3}{2}}) &\text{if }\hspace{1cm}  \begin{rcases} 
                                               -1 \leq \thr_i, \thl_\ipo < 1 \hspace{2.4cm} \text{ (case 6)} \end{rcases} \Omega_3
                                               \end{cases}
\end{equation*}

\begin{lemma}[Bounds on jumps] We have the following estimate on the jump in the SP-WENO reconstruction:
\begin{equation}\label{eqn:stab_bound}
\left| \llbracket v \rrbracket_\iph \right| \leq 2 \left| \Delta v_\iph \right| \quad \forall\ i \in \mathbb{Z}.
\end{equation}
\end{lemma}
\begin{proof}
If $\Delta v_\iph = 0$, then the estimate estimate clearly holds as $\llbracket v \rrbracket_\iph = 0$ by construction of SP-WENO. Thus, we assume $\Delta v_\iph \neq 0$. Furthermore, the above estimate holds trivially for $(\thr_i,\thl_\ipo) \in \Omega_0$. 

If $(\thr_i,\thl_\ipo) \in \Omega_1$, then 
\[
|\thr_i| \leq 1 \quad \iff \quad -1 \leq \frac{\Delta v_\imh}{\Delta v_\iph} \leq 1.
\]
Thus,
\[
\frac{\llbracket v \rrbracket_\iph}{\Delta v_\iph} = \frac{1}{2} - \frac{\Delta v_\imh}{2\Delta v_\iph} \leq 1 \quad \implies \quad \left| \llbracket v \rrbracket_\iph \right| \leq  \left| \Delta v_\iph \right| < 2  \left| \Delta v_\iph \right|
\]
Similarly, if $(\thr_i,\thl_\ipo) \in \Omega_2$, then 
\[
|\thl_\ipo| \leq 1 \quad \iff \quad -1 \leq \frac{ \Delta v_{i + \frac{3}{2}}}{\Delta v_\iph} \leq 1
\]
Thus,
\[
\frac{\llbracket v \rrbracket_\iph}{\Delta v_\iph} = \frac{1}{2} - \frac{ \Delta v_{i + \frac{3}{2}}}{2\Delta v_\iph} \leq 1 \quad \implies \quad \left| \llbracket v \rrbracket_\iph \right| \leq  \left| \Delta v_\iph \right| < 2  \left| \Delta v_\iph \right|
\]
Finally, if $(\thr_i,\thl_\ipo) \in \Omega_3$, then
\[
|\thr_i|,\ |\thl_\ipo|\leq 1
\]
Repeating the above arguments, we once again get
\[
\left| \llbracket v \rrbracket_\iph \right| \leq 2 \left| \Delta v_\iph \right|. 
\]
\end{proof}
\begin{remark}
The bounding constant $2$ on the right-hand side of \eqref{eqn:stab_bound} is identical to the one obtained with ENO-2 \cite{FMT13}, but smaller than that obtained with ENO-3. Thus, SP-WENO leads to tighter stability bounds for higher order accuracy, as compared to its ENO counterparts.
\end{remark}  
\section{Numerical results}\label{sec:numerical_results}
We present several numerical results to demonstrate the performance of the SP-WENO reconstruction. In each test case, a finite domain $x\in[a,b]$ is considered with either periodic or Neumann boundary conditions. An $N$-cell uniform mesh with mesh size $h = \frac{b-a}{N}$ is generated as
\[
x_\iph := a + ih, \qquad\qquad x_i := \frac{x_\imh+x_\iph}{2}.
\]
At times, ghost cells would be required to extend the mesh on either side. For instance, if both $v_\frac{1}{2}^-$ and $v_\frac{1}{2}^+$ are required, then two ghost cells need to be introduced on the left. The value of the solution in the ghost cells can be set depending on the boundary condition.

First, we demonstrate that our reconstruction SP-WENO does indeed give the desired order accuracy in approximating the solution at cell-interfaces, assuming the solution is smooth enough. 

\subsection{Reconstruction accuracy}
We consider the smooth function 
\[
u(x) = \sin{(10 \pi x)} + x, \quad x \in [0,1].
\]
to test the accuracy of SP-WENO. We also compare the results with ENO-2, ENO-3 and the existing robust version of WENO-3 proposed in \cite{JS96} (see also \cite{SHU98}). The error in the interface values are evaluated as
\[
 \| u_\iph^- - u(x_\iph) \|_{L^p_h} +  \| u_\iph^+ - u(x_\iph) \|_{L^p_h},  \quad p \in [1,\infty]
\]
where the discrete norm is defined as
\[
\|(.)_i\|_{L^p_h} =\left( \sum\limits_{i=1}^N | (.)_i |^p h \right)^\frac{1}{p} \quad \text{for } p < \infty, \qquad \|(.)_i\|_{L^\infty_h} =\max \limits_i | (.)_i |
\]

The error in approximating the interface values, and the corresponding convergence rates are shown in Table \ref{tab:inclined_sine}. SP-WENO gives third order accuracy when we consider the $L_h^\infty$ norm. In fact, SP-WENO seems to give more than third order convergence in the $L_h^1$ norm, which is not the case for the ENO reconstructions. An improved superior convergence rate in the $L_h^1$ norm is also seen with WENO-3. Thus, one can expect the SP-WENO to give superior convergence results for evolution problems, in the $L^1$ sense. Looking at the finest mesh ($N=2560$), SP-WENO gives the best results from the point of view of error magnitude. Note that WENO-3 gives almost fourth-order convergence in the $L_h^\infty$ norm. This can possibly be explained by the fact that unlike the SP-WENO, the weights used in WENO-3 are smooth. 

\begin{table}
\begin{center}
\begin{tabular}{|c|c|c|c|c||c|c|c|c|} \hline 
& \multicolumn{4}{|c||} {SP-WENO} & \multicolumn{4}{|c|} {ENO3} \\ \hline 
\multirow{2}{*}{N} & \multicolumn{2}{|c|} {$L^{1}_h$} & \multicolumn{2}{|c||} {$L^\infty_h$} & \multicolumn{2}{|c|} {$L^{1}_h$} & \multicolumn{2}{|c|} {$L^\infty_h$}\\ \cline{2-9} 
& error & rate& error & rate& error & rate& error & rate\\ \hline  
40& 8.59e-02 &  - & 2.24e-01 &  - & 3.95e-02 & - & 5.60e-02 &  -\\ \cline{1-5} \cline{6-9} 
80& 6.73e-03 &  3.67& 2.97e-02 &  2.92& 4.90e-03 &  3.01& 7.43e-03 &  2.92\\ \cline{1-5} \cline{6-9} 
160& 5.01e-04 &  3.75& 3.77e-03 &  2.98& 6.08e-04 &  3.01& 9.42e-04 &  2.98\\ \cline{1-5} \cline{6-9} 
320& 3.64e-05 &  3.78& 4.73e-04 &  2.99& 7.57e-05 &  3.01& 1.18e-04 &  2.99\\ \cline{1-5} \cline{6-9} 
640& 2.59e-06 &  3.81& 5.91e-05 &  3.00& 9.47e-06 &  3.00& 1.48e-05 &  3.00\\ \cline{1-5} \cline{6-9} 
1280& 1.82e-07 &  3.83& 7.39e-06 &  3.00& 1.18e-06 &  3.01& 1.85e-06 &  3.00\\ \cline{1-5} \cline{6-9} 
2560& 1.26e-08 &  3.85& 9.24e-07 &  3.00& 1.47e-07 &  3.00& 2.31e-07 &  3.00\\ \hline 
\end{tabular} 
\vspace{0.5mm}
\begin{tabular}{|c|c|c|c|c||c|c|c|c|} \hline 
& \multicolumn{4}{|c||} {WENO3} & \multicolumn{4}{|c|} {ENO2} \\ \hline 
\multirow{2}{*}{N} & \multicolumn{2}{|c|} {$L^{1}_h$} & \multicolumn{2}{|c||} {$L^\infty_h$} & \multicolumn{2}{|c|} {$L^{1}_h$} & \multicolumn{2}{|c|} {$L^\infty_h$}\\ 
\cline{2-9} 
& error & rate& error & rate& error & rate& error & rate\\ \hline 
40& 2.04e-01 & -& 4.34e-01 & -& 2.35e-01 & -& 4.34e-01 & -\\ \hline  
80& 4.03e-02 &  2.34& 1.14e-01 &  1.93& 5.39e-02 &  2.12& 1.14e-01 &  1.93\\ \hline 
160& 7.25e-03 &  2.48& 2.88e-02 &  1.98& 1.29e-02 &  2.07& 2.88e-02 &  1.98\\ \hline 
320& 1.18e-03 &  2.62& 7.10e-03 &  2.02& 3.14e-03 &  2.03& 7.22e-03 &  2.00\\ \hline 
640& 1.77e-04 &  2.74& 1.65e-03 &  2.10& 7.76e-04 &  2.02& 1.81e-03 &  2.00\\ \hline  
1280& 2.13e-05 &  3.05& 1.64e-04 &  3.34& 1.93e-04 &  2.01& 4.52e-04 &  2.00\\ \hline 
2560& 2.10e-06 &  3.34& 9.08e-06 &  4.17& 4.81e-05 &  2.00& 1.13e-04 &  2.00\\ \hline 
\end{tabular} 
\caption{Inclined sine wave: reconstruction errors.}
\label{tab:inclined_sine}
\end{center}
\end{table}

%

\subsection{Evolution problems}
We now test the performance of SP-WENO in solving one-dimensional scalar conservation laws. We use the TeCNO4 (finite difference) numerical flux
\[
f_\iph = f^*_\iph - \frac{1}{2} a_\iph \left (v^+_\iph - v^-_\iph \right)
\]
where $f^*_\iph$ is the fourth-order accurate entropy conservative flux given by \eqref{eqn:EC4}, while $v:=\eta'(u)=u$ is the entropy variable for the square entropy $\eta(u)=\frac{u^2}{2}$. The reconstruction is performed using SP-WENO, ENO2 or ENO3, all of which have the sign property. Time integration is performed using SSP-RK3 (see \cite{GOT01}).

\subsubsection{Linear Advection}
We first look at the linear advection equation
\[
u_t + c u_x = 0.
\]
The entropy conservative component of the flux is chosen using \eqref{eqn:ecflux_linadv} and \eqref{eqn:EC4}, while $a_\iph = |c|$. For the following test cases we take the convective velocity $c=1$.

\noindent
\textbf{Test 1:} The domain is $[ -\pi, \pi]$,  final time is $T = 0.5$ and CFL = 0.4, with the initial profile given by
\[
u_0(x) = \sin{(x)}
\]
and periodic boundary conditions. Table~\ref{table:linadv_test1} shows $L^1_h$ errors with various reconstructions. SP-WENO gives more than third-order accuracy, while ENO2 and ENO3 give expected convergence rates. 

\begin{table}
\begin{center}
\begin{tabular}{|c|c|c||c|c||c|c|} \hline 
& \multicolumn{2}{|c||} {SP-WENO} & \multicolumn{2}{|c||} {ENO3} & \multicolumn{2}{|c|} {ENO2} \\ \hline 
\multirow{2}{*}{N} & \multicolumn{2}{|c||} {$L^{1}_h$}  & \multicolumn{2}{|c||} {$L^{1}_h$} & \multicolumn{2}{|c|} {$L^{1}_h$}\\ \cline{2-7} 
& error & rate& error & rate& error & rate\\ \hline 
50& 6.22e-04& - & 2.58e-04 & - & 1.61e-02& -\\ \hline 
100& 6.90e-05& 3.17 & 3.23e-05& 3.00 & 4.36e-03& 1.88\\ \hline 
200& 7.66e-06& 3.17 & 4.04e-06& 3.00 & 1.16e-03& 1.91\\ \hline 
400& 8.29e-07& 3.21 & 5.05e-07& 3.00 & 3.08e-04& 1.91\\ \hline 
600& 2.26e-07& 3.20 & 1.50e-07& 3.00 & 1.41e-04& 1.92\\ \hline 
800& 8.72e-08& 3.31 & 6.31e-08& 3.00 & 8.09e-05& 1.93\\ \hline 
\end{tabular}
\caption{Linear advection smooth test 1 with TeCNO4 flux.}
\label{table:linadv_test1}
\end{center}
\end{table}

\noindent
\textbf{Test 2:} The domain is $[ -\pi, \pi]$,  final time is $T = 0.5$ and CFL = 0.5, with the initial profile given by
\[
u_0(x) = \sin^4{(x)}
\]
and periodic boundary conditions. The MUSCL scheme using ENO reconstruction is known to perform poorly for this test case; see Rogerson and Meiburg \cite{ROGER90}. This is also observed with the TeCNO4 flux using ENO, as shown in Table \ref{table:linadv_test2}. There is a clear deterioration in the convergence rate for ENO with mesh refinement. SP-WENO, on the other hand, does not suffer from such problems and continues to give more than third order accuracy.

\begin{table}
\begin{center}
\begin{tabular}{|c|c|c||c|c||c|c|} \hline 
& \multicolumn{2}{|c||} {SP-WENO} & \multicolumn{2}{|c||} {ENO3} & \multicolumn{2}{|c|} {ENO2} \\ \hline 
\multirow{2}{*}{N} & \multicolumn{2}{|c||} {$L^{1}_h$} & \multicolumn{2}{|c||} {$L^{1}_h$} & \multicolumn{2}{|c|} {$L^{1}_h$}\\ \cline{2-7} 
& error & rate& error & rate& error & rate\\ \hline 
100& 1.32e-03 & - & 1.48e-03 & - & 2.13e-02 & -\\ \hline 
200& 1.48e-04 &  3.16 & 1.97e-04 &  2.91 & 6.12e-03 &  1.80\\ \hline 
400& 1.64e-05 &  3.17 & 2.57e-05 &  2.94 & 1.66e-03 &  1.89\\ \hline 
600& 4.61e-06 &  3.14 & 8.35e-06 &  2.77 & 7.63e-04 &  1.91\\ \hline 
800& 1.79e-06 &  3.29 & 4.86e-06 &  1.88 & 4.41e-04 &  1.90\\ \hline 
1000& 8.55e-07 &  3.31 & 3.62e-06 &  1.32 & 2.87e-04 &  1.92\\ \hline  
\end{tabular}
\caption{Linear advection smooth test 2 with TeCNO4 flux.}
\label{table:linadv_test2}
\end{center}
\end{table}

\noindent
\textbf{Test 3:} The domain is $[ -1, 1]$,  final time is $T = 0.5$ and CFL = 0.4, with the initial profile being discontinuous:
\[
u_0(x) = \begin{cases} 
	3 &\text{if } x < 0 \\
	-1 &\text{if } x > 0.
\end{cases}
\]
The mesh consists of 100 cells with Neumann boundary conditions. The results with the TeCNO4 scheme are show in Figure~\ref{fig:linadv_test3}(a). While ENO-2 and ENO-3 reconstruction seem to give oscillation-free solutions, SP-WENO leads to minor undershoots and overshoots near the discontinuity. The solutions with mesh refinement for SP-WENO are shown in Figure \ref{fig:linadv_test3}(b).
\begin{figure}
\begin{center}
\subfigure[Comparison of different reconstruction methods]{\includegraphics[width=0.49\textwidth]{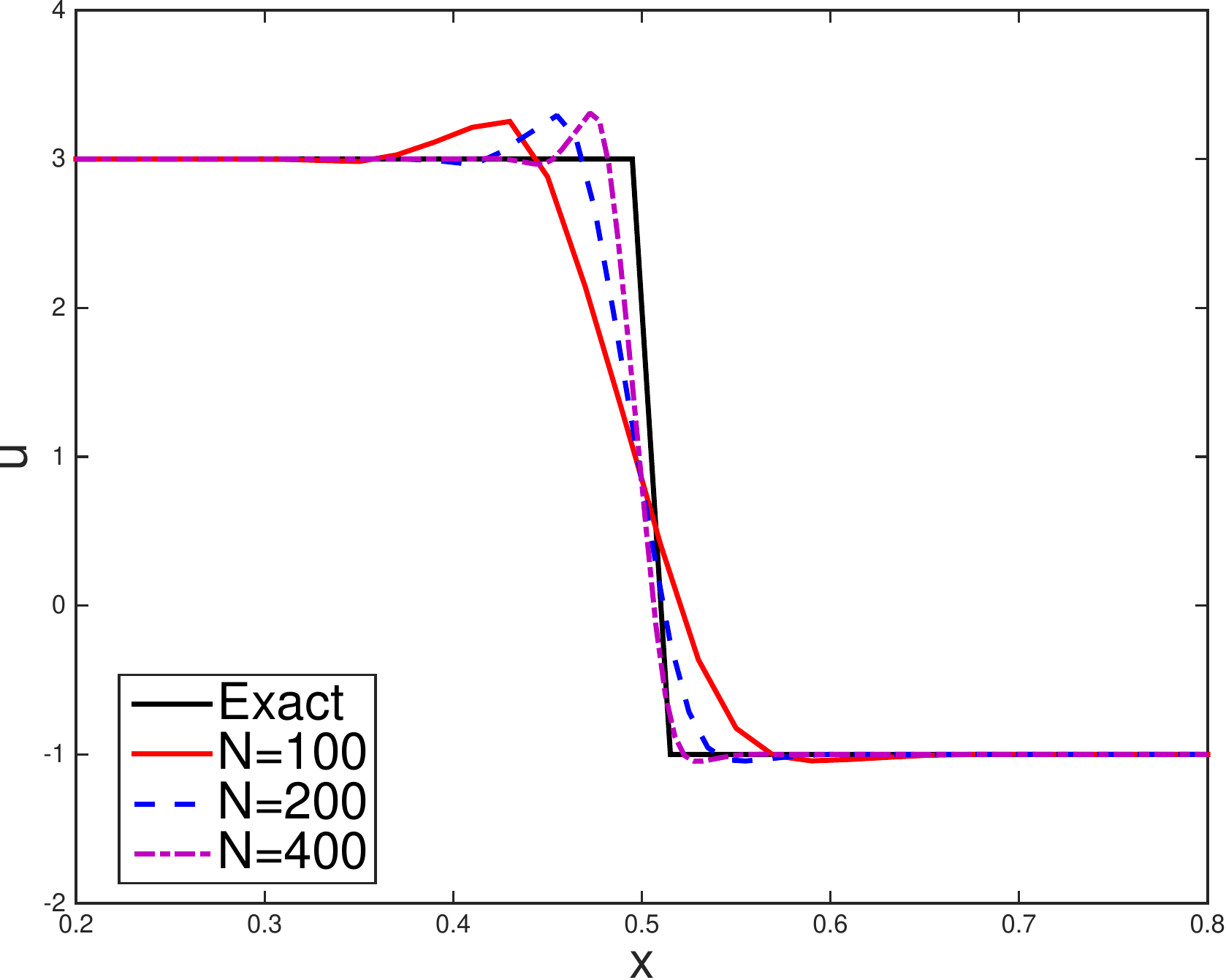}}
\subfigure[Mesh refinement study with SP-WENO]{\includegraphics[width=0.49\textwidth]{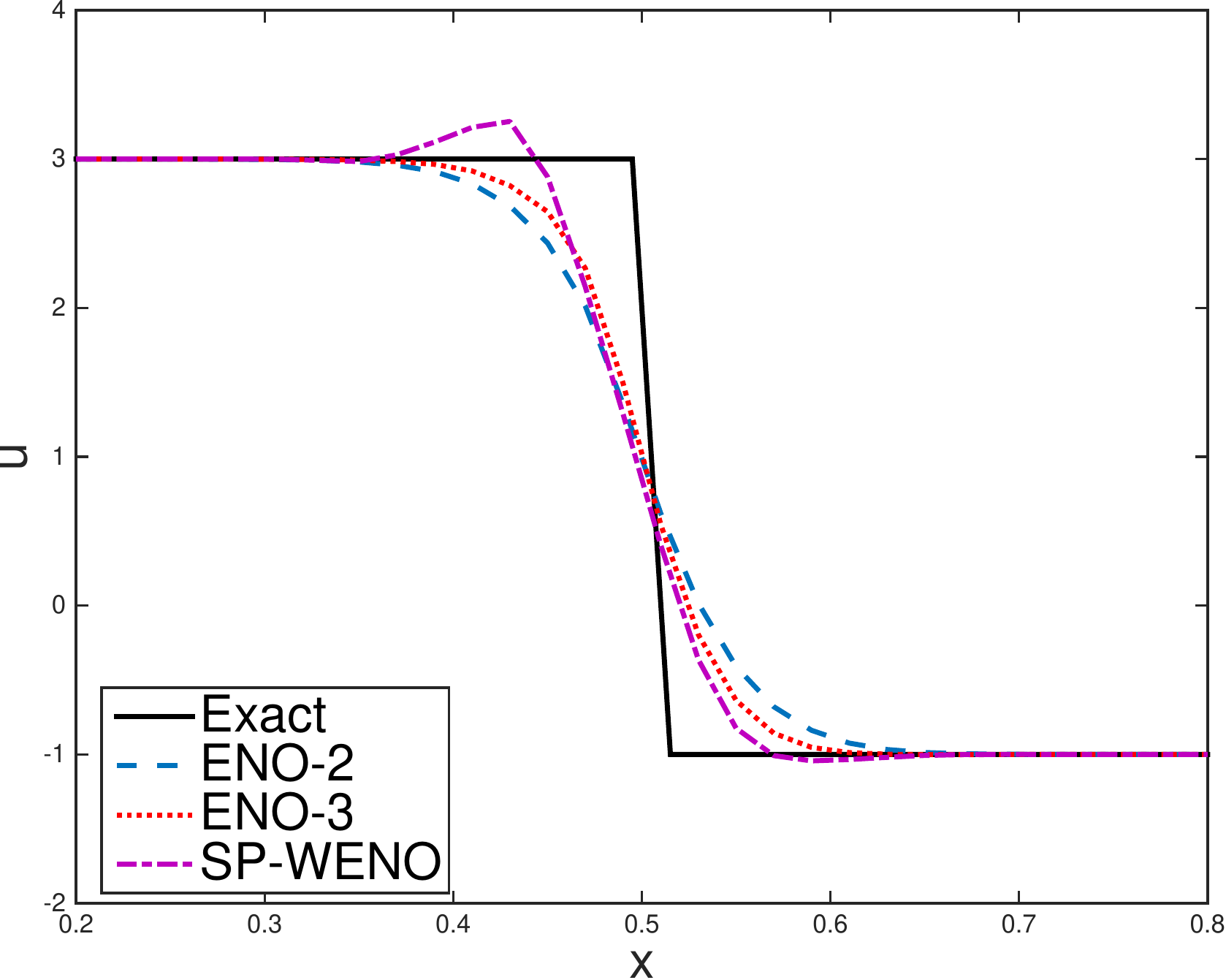}}
\caption{Linear advection test 3: Solution with TeCNO4 at time T=0.5 for $x\in[0.2,0.8]$.}
\label{fig:linadv_test3}
\end{center}
\end{figure}

\subsection{Burgers' Equation}
Next, we consider the Burgers equation
\[
u_t + \left(\frac{u^2}{2} \right)_x = 0
\]
with the entropy conservative component of the flux chosen as \eqref{eqn:ecflux_burgers}, \eqref{eqn:EC4}. We use the diffusion coefficient 
\[
a_\iph = \frac{|u_i| + |u_\ipo|}{2}.
\]

\noindent
\textbf{Test 1:} The domain is $[ -1, 1]$,  final time is $T = 0.3$ and CFL = 0.4, with the initial profile given by
\[
u_0(x) = 1+ \frac{1}{2} \sin{(\pi x)}
\]
and periodic boundary conditions. Table~\ref{table:burgers_test1}  clearly show a deterioration of TeCNO4 with ENO-3 reconstruction, while SP-WENO seems to once again give more than third order accuracy. 

Being a non-linear problem, the solution develops a discontinuity in finite time, which can be evaluated to be $t = \frac{2}{\pi} \approx 0.636$. The total entropy of the problem should be preserved over time as long as the solution is smooth. However, the time-stepping scheme introduces a small amount numerical diffusion. After the appearance of the discontinuity, a sharp decrease in total entropy is expected. To see this, we evaluate the relative change in total entropy 
\[
\frac{E(t) - E(0)}{E(0)}, \qquad E(t) := \sum \limits_i \eta_i(t) h \approx \int\limits_{-1}^{1} \eta(u(x,t)) \ud x
\]
up to time $T=0.7$. The results shown in Figure \ref{fig:rel_ent} clearly show that SP-WENO performs the best from the point of view of preservation of global entropy, prior to the shock. The most dissipative solutions are obtained with ENO-2, while ENO-3 lies somewhere in between. Moreover, the performance with all reconstructions improves with mesh refinement.

\begin{table}
\begin{center}
\begin{tabular}{|c|c|c||c|c||c|c|} \hline 
& \multicolumn{2}{|c||} {SP-WENO} & \multicolumn{2}{|c||} {ENO3} & \multicolumn{2}{|c|} {ENO2} \\ \hline 
\multirow{2}{*}{N}  & \multicolumn{2}{|c||} {$L^{1}_h$} & \multicolumn{2}{|c||} {$L^{1}_h$} & \multicolumn{2}{|c|} {$L^{1}_h$}\\ \cline{2-7} 
& error & rate& error & rate& error & rate\\ \hline 
50& 3.41e-04 & - & 3.07e-04 & - & 4.73e-03 & -\\ \hline 
100& 4.17e-05 &  3.03 & 4.76e-05 &  2.69 & 1.35e-03 &  1.81\\ \hline 
200& 4.51e-06 &  3.21 & 8.44e-06 &  2.49 & 3.77e-04 &  1.84\\ \hline 
400& 4.98e-07 &  3.18 & 1.80e-06 &  2.23 & 1.02e-04 &  1.89\\ \hline 
600& 1.33e-07 &  3.26 & 7.29e-07 &  2.23 & 4.71e-05 &  1.90\\ \hline 
800& 5.22e-08 &  3.25 & 3.91e-07 &  2.17 & 2.72e-05 &  1.92\\ \hline 
\end{tabular}
\caption{Burgers equation smooth test 1.}
\label{table:burgers_test1}
\end{center}
\end{table}

\begin{figure}
\begin{center}
\includegraphics[width=0.55\textwidth]{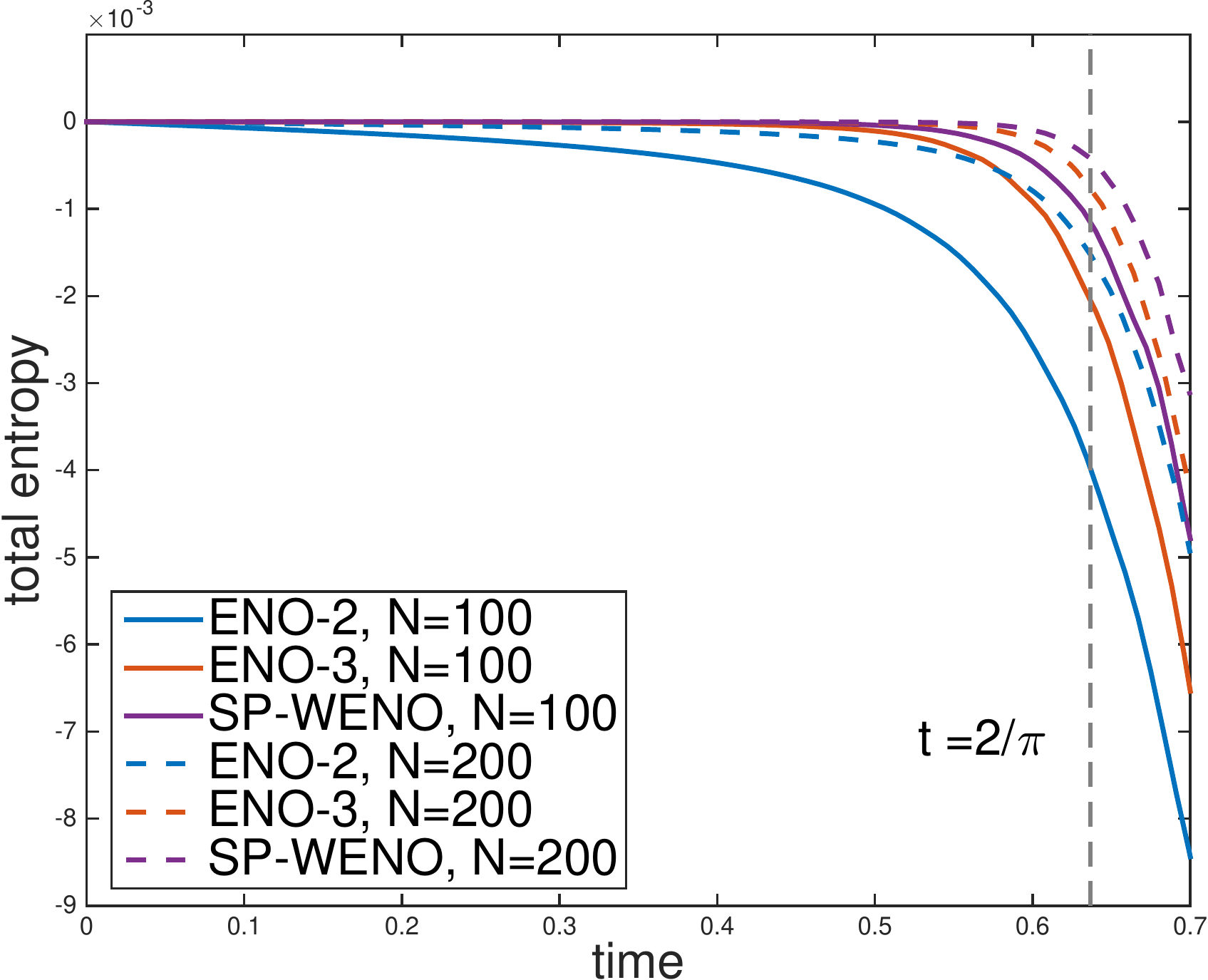}
\caption{Burgers equation smooth test 1: relative change in total entropy.}
\label{fig:rel_ent}
\end{center}
\end{figure}

\noindent
\textbf{Test 2:} The domain is $[ -1, 1]$,  final time is $T = 0.45$ and CFL = 0.4, with the initial profile given by
\[
u_0(x) = \begin{cases} 3, \quad \text {if } x < 0 \\
                                   -1, \quad \text {if } x > 0
              \end{cases}                     
\]
which corresponds to a left-moving shock. The mesh consists of 100 cells with Neumann boundary conditions. The solutions with the TeCNO4 flux is shown in Figure \ref{fig:burgers_test2}. As can be seen, ENO-2, ENO-3 and SP-WENO gives minor oscillations near the shock. The shock is equally well resolved by each method.

\begin{figure}
\begin{center}
\subfigure[Solution with TeCNO4 at time $t=0.45$]{\includegraphics[width=0.45\textwidth]{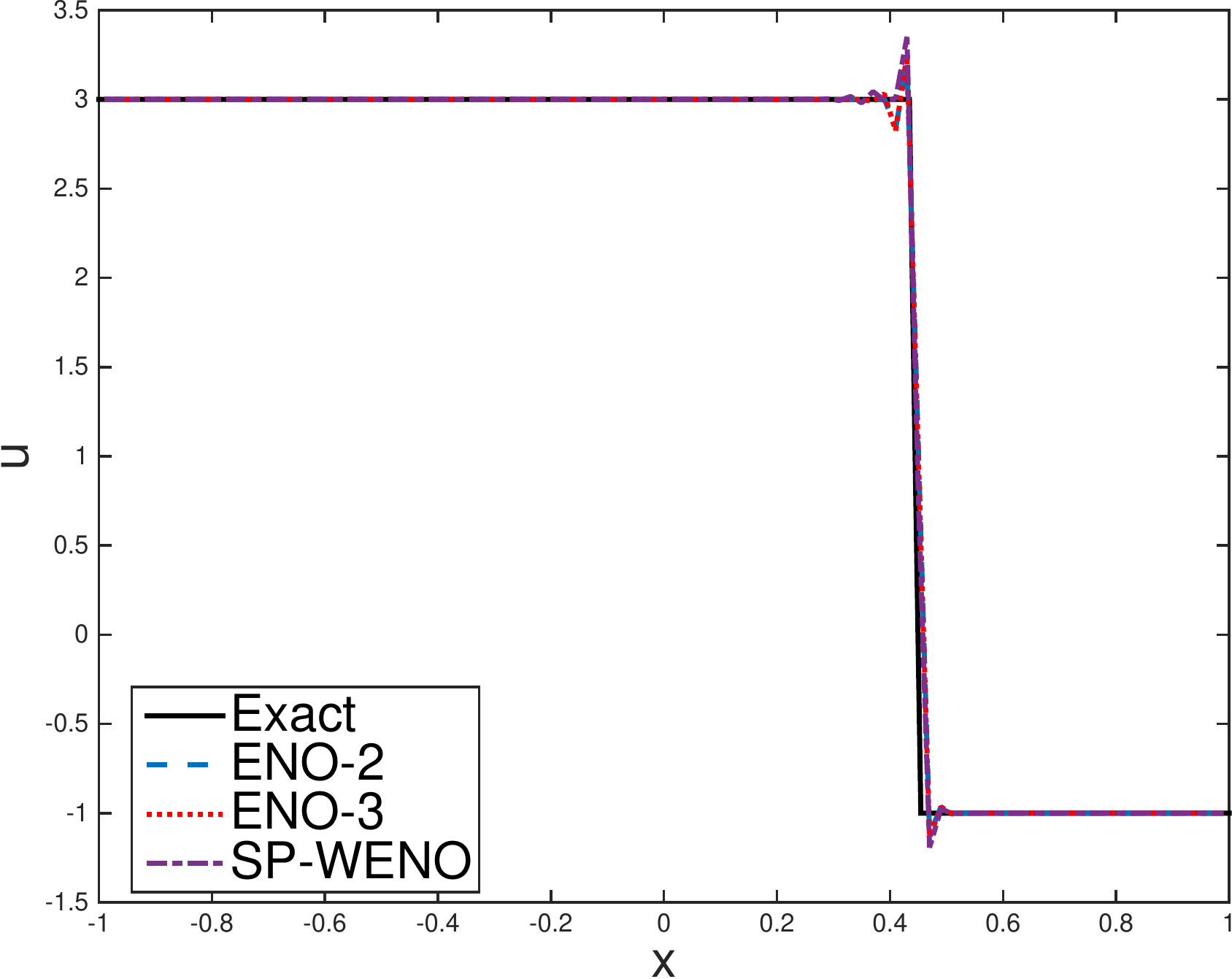}}
\subfigure[Zoomed pre-shock region]{\includegraphics[width=0.45\textwidth]{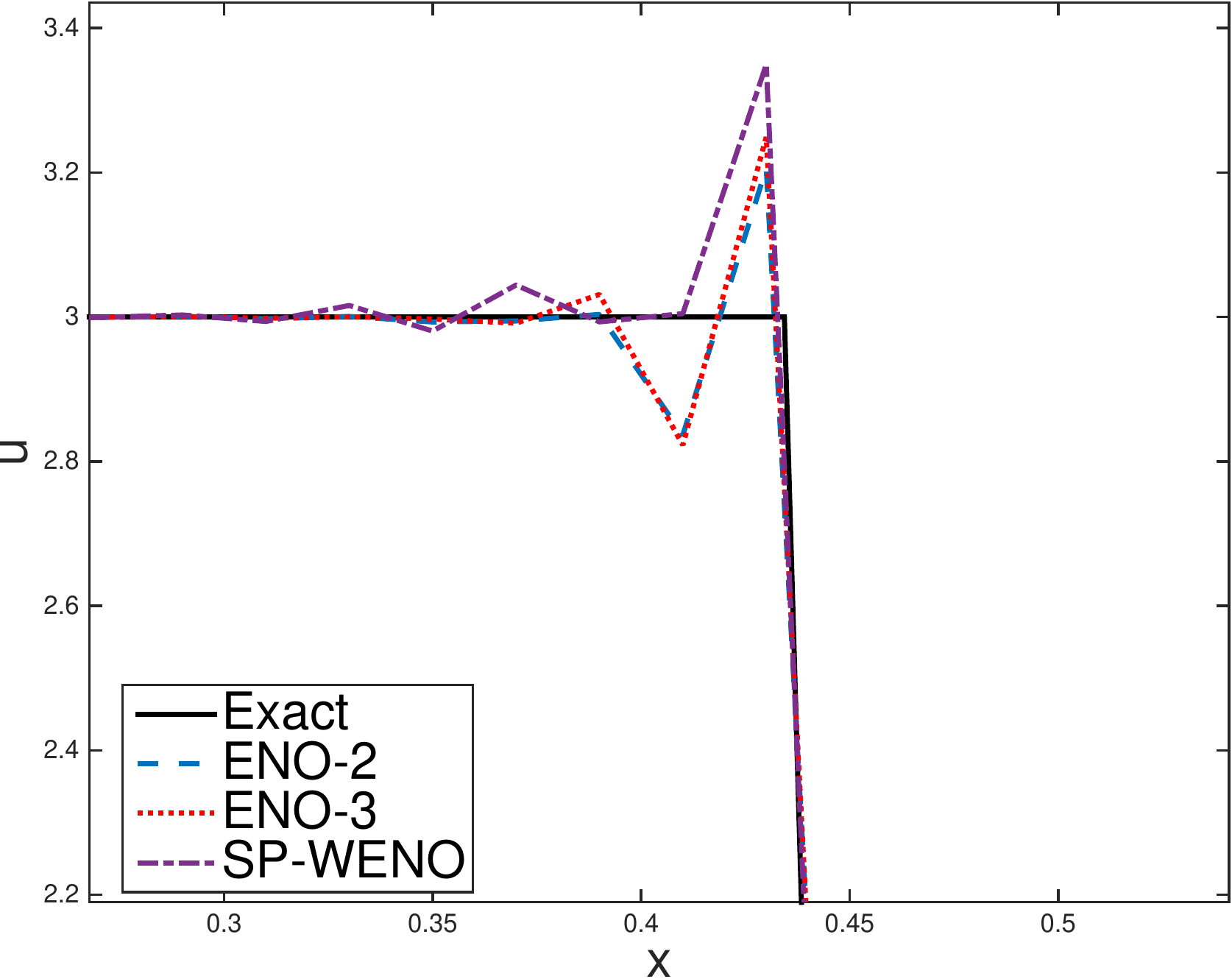}}
\caption{Burgers equation test 2.}
\label{fig:burgers_test2}
\end{center}
\end{figure}

\noindent
\textbf{Test 3:} The domain is $[ -1, 1]$,  final time is $T = 0.2$ and CFL = 0.4, with the initial profile given by
\[
u_0(x) = \begin{cases} -2, \quad \text {if } x < 0 \\
                                   1, \quad \text {if } x > 0
              \end{cases}                     
\]
which corresponds to a rarefaction wave. The mesh consists of 100 cells with Neumann boundary conditions. The solutions shown in Figure \ref{fig:burgers_test3} indicates that SP-WENO gives the sharpest solution, while ENO-2 is the most dissipative.

\begin{figure}
\begin{center}
\subfigure[]{\includegraphics[width=0.45\textwidth]{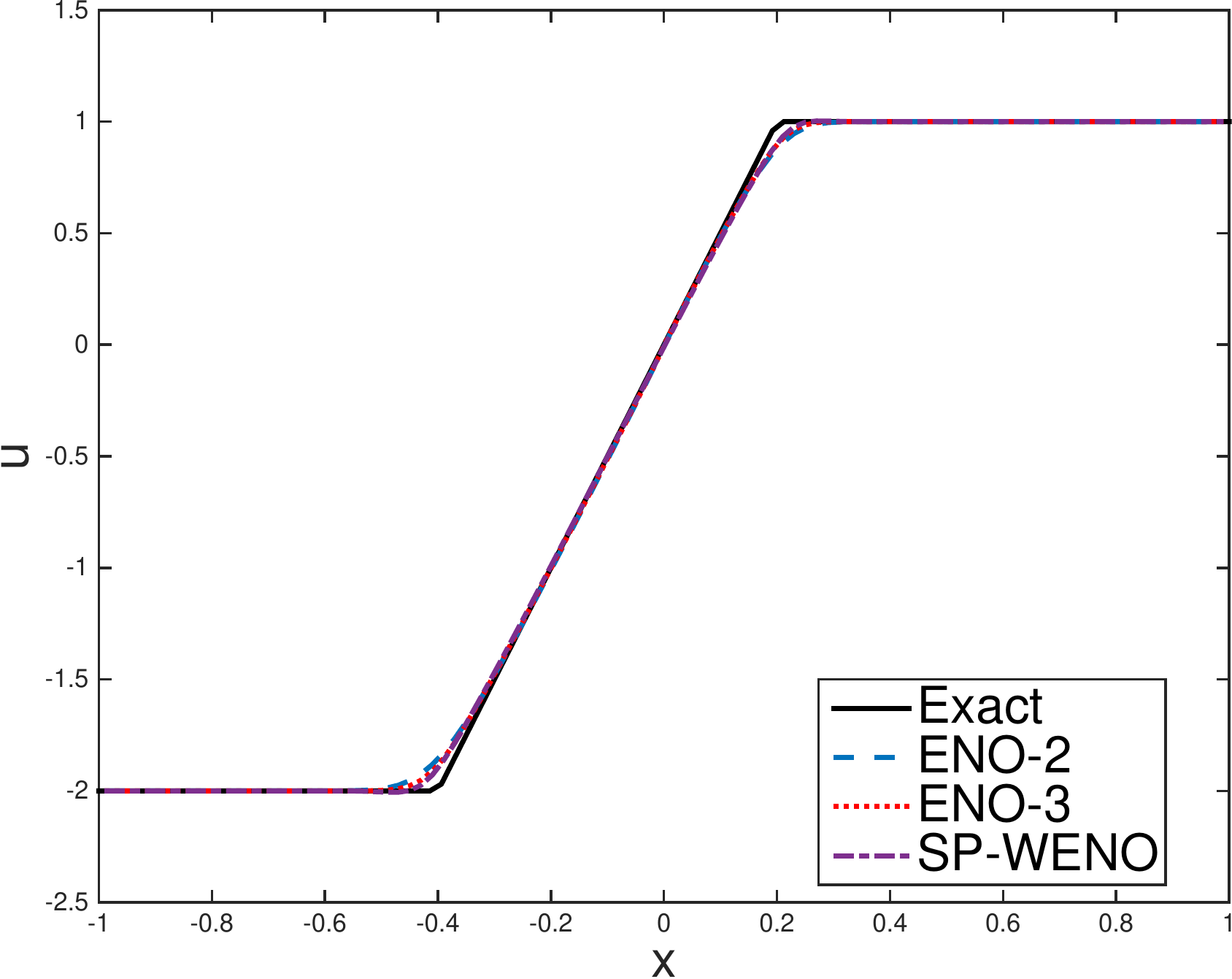}}
\subfigure[]{\includegraphics[width=0.45\textwidth]{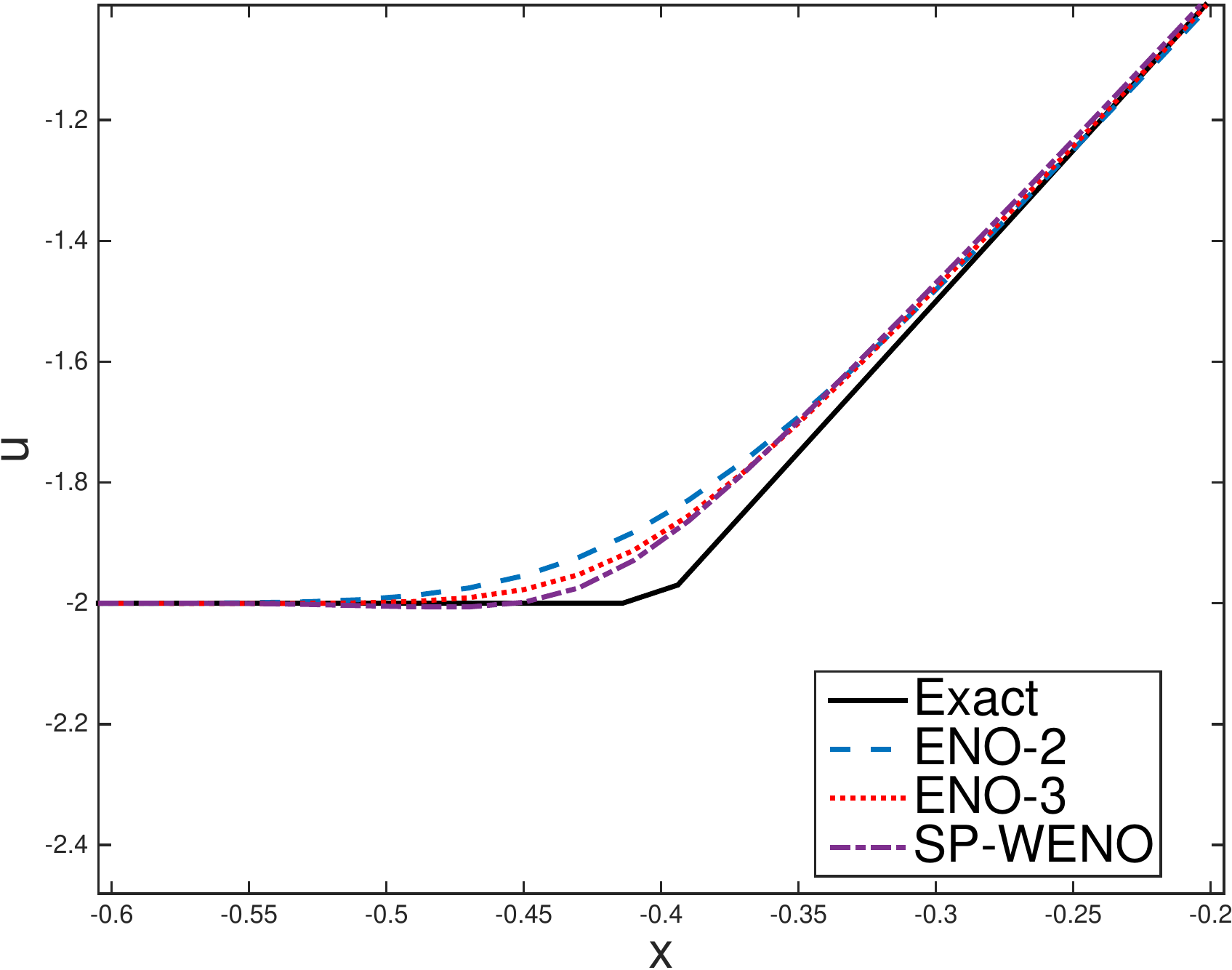}}
\caption{Burgers equation test 3: Solution with TeCNO4 at time $t=0.2$.}
\label{fig:burgers_test3}
\end{center}
\end{figure}

\begin{remark}
The minor oscillations visible in Figure \ref{fig:burgers_test3} can be attributed to insufficient dissipation near shocks. One possibility of improving the solution would be to modify the dissipation operator to obtain {\it entropy consistent} schemes, as described in \cite{ROE09}. However, this would lead to the introduction of an additional tuning parameter into the scheme, which we intend to avoid. Note that despite the presence of oscillations, the scheme satisfies the entropy condition and thus, by a Lax-Wendroff-type argument, is guaranteed to converge to the unique entropy solution whenever it converges.
\end{remark}

\section{Conclusion}\label{sec:conclusion}
The entropy condition ensures stability and uniqueness of weak solutions for hyperbolic conservation laws, and in order to converge to the entropy solution, numerical methods for hyperbolic conservation laws must be consistent with this entropy condition. It was found in \cite{FMT12} that reconstruction-based, high-order finite difference and finite volume methods are entropy stable provided the reconstruction method satisfies the \emph{sign property}. Specifically, a class of schemes called the TeCNO schemes were designed, relying crucially on the sign property to ensure that the numerical diffusion coefficient has the right (positive) sign. It was shown in \cite{FMT13} that the ENO reconstruction method satisfies this property.

It is well-known that the WENO reconstruction methods exhibit several advantages over ENO; see e.g.\ \cite{SHU98}. In the present paper we design a third-order WENO reconstruction method which satisfies the sign property. We define the \emph{feasible region} as the set of all stencil weights satisfying properties \eqref{eqn:cons_a}--\eqref{eqn:signprop}, and then choose the point in this region such that the order contraints given by \eqref{eqn:variableorder} are satisfied (for smooth solutions). Furthermore, the weights are constructed to satisfy several symmetry and consistency conditions, (properties \eqref{eqn:negsym}--\eqref{eqn:mirsym}).  In a series of numerical experiments we demonstrate that the TeCNO scheme, using the proposed WENO reconstruction, is both third (or higher) order accurate and entropy stable, both for linear and nonlinear problems.

The SP-WENO reconstruction can be used to construct higher order entropy stable schemes for systems of conservation laws in one-dimension, in which case {\it scaled entropy variables} must be reconstructed \cite{FMT12}. Furthermore, being a finite difference scheme, the scheme can be easily extended to higher-dimensional problems on Cartesian meshes, by treating it in a dimension by dimension fashion, as described e.g.\ in \cite[Section 6]{FMT12}. The possibility of extending the proposed ideas to construct sign-preserving WENO schemes of fifth-order or higher are being currently investigated, and will be the topic of future papers.


\begin{thebibliography}{10}

\bibitem{HARTEN84}
A. Harten.
\newblock On a class of high resolution total-variation-stable
  finite-difference schemes.
\newblock {\em SIAM J. Numer. Anal.}, 21(1):1--23, 1984.

\bibitem{CRANDALL80}
M. G. Crandall and A. Majda.
\newblock Monotone difference approximations for scalar conservation laws.
\newblock {\em Math. Comp.}, 34(149):1--21, 1980.


\bibitem{DAFER10}
C. M. Dafermos.
\newblock {\em Hyperbolic conservation laws in continuum physics}, volume 325
  of {\em Grundlehren der Mathematischen Wissenschaften}.
\newblock Springer-Verlag, Berlin, third edition, 2010.

\bibitem{FMT12}
U. S. Fjordholm, S. Mishra, and E. Tadmor.
\newblock Arbitrarily high-order accurate entropy stable essentially
  nonoscillatory schemes for systems of conservation laws.
\newblock {\em SIAM Journal on Numerical Analysis}, 50(2):544--573, 2012.

\bibitem{FMT13}
U. S. Fjordholm. S. Mishra and E. Tadmor.
\newblock ENO reconstruction and ENO interpolation are stable.
\newblock {\em FoCM} 13 (2), 2013, 139--159.

\bibitem{GOT01}
S. Gottlieb, C-W. Shu, and E. Tadmor.
\newblock Strong stability-preserving high-order time discretization methods.
\newblock {\em SIAM Rev.}, 43(1):89--112, 2001.

\bibitem{HARTEN87}
A. Harten, B. Engquist, S. Osher, and S. R. Chakravarthy.
\newblock Uniformly high order accurate essentially non-oscillatory schemes, {III}.
\newblock {\em Journal of Computational Physics}, 71(2):231 -- 303, 1987.

\bibitem{ROE09}
Farzad Ismail and Philip~L. Roe.
\newblock Affordable, entropy-consistent euler flux functions {II}: Entropy
  production at shocks.
\newblock {\em Journal of Computational Physics}, 228(15):5410 -- 5436, 2009.

\bibitem{JS96}
G.-S. Jiang and C.-W. Shu.
\newblock Efficient implementation of weighted {ENO} schemes.
\newblock \emph{J. Comput. Phys.} 126(1), Academic Press Professional, Inc., 202–228, 1996.

\bibitem{KRUZ70}
S.~N. Kru{\v{z}}kov.
\newblock First order quasilinear equations with several independent variables.
\newblock {\em Mat. Sb. (N.S.)}, 81 (123):228--255, 1970.

\bibitem{LMR02}
P.~G. Lefloch, J.~M. Mercier, and C.~Rohde.
\newblock Fully discrete, entropy conservative schemes of arbitrary order.
\newblock {\em SIAM J. Numer. Anal.}, 40(5):1968--1992, 2002.

\bibitem{LIU94}
X.-D. Liu, S. Osher, and T. Chan.
\newblock Weighted essentially non-oscillatory schemes.
\newblock {\em Journal of Computational Physics}, 115(1):200 -- 212, 1994.

\bibitem{OSHER84a}
S. Osher.
\newblock Riemann solvers, the entropy condition, and difference
  approximations.
\newblock {\em SIAM J. Numer. Anal.}, 21(2):217--235, 1984.

\bibitem{OSHER84b}
S. Osher and S. Chakravarthy.
\newblock High resolution schemes and the entropy condition.
\newblock {\em SIAM J. Numer. Anal.}, 21(5):955--984, 1984.

\bibitem{OSHER88}
S. Osher and E. Tadmor.
\newblock On the convergence of difference approximations to scalar
  conservation laws.
\newblock {\em Math. Comp.}, 50(181):19--51, 1988.

\bibitem{ROGER90}
A. M. Rogerson and E.~Meiburg.
\newblock A numerical study of the convergence properties of ENO schemes.
\newblock {\em Journal of Scientific Computing}, 5(2):151--167, 1990.


\bibitem{SHU98}
C.-W. Shu.
\newblock Essentially non-oscillatory and weighted essentially non-oscillatory
  schemes for hyperbolic conservation laws.
\newblock In Alfio Quarteroni, editor, {\em Advanced Numerical Approximation of
  Nonlinear Hyperbolic Equations}, volume 1697 of {\em Lecture Notes in
  Mathematics}, pages 325--432. Springer, 1998.
  
\bibitem{SWEBY84}
P.~K. Sweby.
\newblock High resolution schemes using flux limiters for hyperbolic
  conservation laws.
\newblock {\em SIAM J. Numer. Anal.}, 21(5):995--1011, 1984.  
  
\bibitem{TADMOR84}
E.~Tadmor.
\newblock {Numerical Viscosity and the Entropy Condition for Conservative
  Difference Schemes}.
\newblock {\em Mathematics of Computation}, 43(168):369--381, 1984.
  
\bibitem{TADMOR03}
E. Tadmor.
\newblock Entropy stability theory for difference approximations of nonlinear
  conservation laws and related time-dependent problems.
\newblock {\em Acta Numerica}, 12:451--512, 2003.

\end{thebibliography}
\end{document}